%% file: paper_aos.tex
\documentclass[aos]{imsart}

\RequirePackage{amsthm,amsmath,amsfonts,amssymb}
\RequirePackage[numbers,sort&compress]{natbib}
\RequirePackage[colorlinks,citecolor=blue,urlcolor=blue]{hyperref}
\RequirePackage{graphicx}
\usepackage{booktabs}
\usepackage{comment}
\startlocaldefs
\theoremstyle{plain}

\theoremstyle{remark}

\endlocaldefs

\usepackage{enumerate}
\usepackage{multicol}
\usepackage{multirow}
\usepackage{fancyhdr}
\usepackage{epsfig}

\usepackage{xparse}

\usepackage[english]{babel}															

\usepackage{multirow}
\usepackage{caption}

\input{definition.tex}

\newcommand{\sumig}{\sum_{i=1}^g}
\newcommand{\sumjh}{\sum_{j=1}^h}
\newcommand{\sumkm}{\sum_{k=1}^m}
\newcommand{\yddd}{\bar{y}}
\newcommand{\yidd}{\bar{y}_{i.}}
\newcommand{\yijd}{\bar{y}_{ij}}
\newcommand{\ydjd}{\bar{y}_{.j}}
\newcommand{\xddd}{\bar{x}}
\newcommand{\xidd}{\bar{x}_{i.}}
\newcommand{\xijd}{\bar{x}_{ij}}
\newcommand{\xdjd}{\bar{x}_{.j}}
\newcommand{\eddd}{\bar{e}}
\newcommand{\eidd}{\bar{e}_{i.}}
\newcommand{\eijd}{\bar{e}_{ij}}
\newcommand{\edjd}{\bar{e}_{.j}}
\newcommand{\gammaid}{\bar{\gamma}_{i.}}
\newcommand{\gammadj}{\bar{\gamma}_{.j}}
\newcommand{\gammadd}{\bar{\gamma}}
\newcommand{\bxai}{\bx_{i}^{(a)}}
\newcommand{\bxait}{\bx_{i}^{(a)T}}
\newcommand{\bxbj}{\bx_{j}^{(b)}}
\newcommand{\bxbjt}{\bx_{j}^{(b)T}}
\newcommand{\bxabij}{\bx_{ij}^{(ab)}}
\newcommand{\bxabijt}{\bx_{ij}^{(ab)T}}
\newcommand{\bxwijk}{\bx_{ijk}^{(w)}}
\newcommand{\bxwijkt}{\bx_{ijk}^{(w)T}}

\newcommand{\barbxa}{\bar{\bx}^{(a)}}
\newcommand{\barbxb}{\bar{\bx}^{(b)}}
\newcommand{\barbxiab}{\bar{\bx}_{i.}^{(ab)}}
\newcommand{\barbxjab}{\bar{\bx}_{.j}^{(ab)}}
\newcommand{\barbxab}{\bar{\bx}^{(ab)}}
\newcommand{\barbxiw}{\bar{\bx}_{i.}^{(w)}}
\newcommand{\barbxjw}{\bar{\bx}_{.j}^{(w)}}
\newcommand{\barbxijw}{\bar{\bx}_{ij}^{(w)}}
\newcommand{\barbxw}{\bar{\bx}_{}^{(w)}}

\newcommand{\barbxat}{\bar{\bx}^{(a)T}}
\newcommand{\barbxbt}{\bar{\bx}^{(b)T}}

\newcommand{\barbxabt}{\bar{\bx}^{(ab)T}}

\newcommand{\barbxwt}{\bar{\bx}_{}^{(w)T}}

\newcommand{\bXa}{\bX^{(a)}}
\newcommand{\bXb}{\bX^{(b)}}
\newcommand{\bXab}{\bX^{(ab)}}
\newcommand{\bXw}{\bX^{(w)}}

\NewDocumentCommand{\barRia}{t^}{%
\IfBooleanTF{#1}
{\barRiaAux}
{\bar{R}_{i.}^{(a)}}%
}
\NewDocumentCommand{\barRiaAux}{m}{%
\bar{R}_{i.}^{(a)#1}%
}

\NewDocumentCommand{\barRjb}{t^}{%
\IfBooleanTF{#1}
{\barRjbAux}
{\bar{R}_{.j}^{(b)}}%
}
\NewDocumentCommand{\barRjbAux}{m}{%
\bar{R}_{.j}^{(b)#1}%
}

\NewDocumentCommand{\barRijab}{t^}{%
\IfBooleanTF{#1}
{\barRijabAux}
{\bar{R}_{ij}^{(ab)}}%
}
\NewDocumentCommand{\barRijabAux}{m}{%
\bar{R}_{ij}^{(ab)#1}%
}

\NewDocumentCommand{\barRijkw}{t^}{%
\IfBooleanTF{#1}
{\barRijkwAux}
{{R}_{ijk}^{(w)}}%
}
\NewDocumentCommand{\barRijkwAux}{m}{%
{R}_{ijk}^{(w)#1}%
}

\newcommand{\sigmacom}{\dot{\tau}}

\newcommand{\hsiga}{\hat{\sigma}_{\alpha}}  
\newcommand{\hsigb}{\hat{\sigma}_{\beta}}

\begin{document}

\begin{frontmatter}
\title{Increasing dimension asymptotics for two-way crossed mixed effect models}
\runtitle{Asymptotics for two-way crossed mixed effect models}

\begin{aug}
\author[A]{\fnms{Ziyang}~\snm{Lyu}\ead[label=e1]{Ziyang.Lyu@unsw.edu.au}\orcid{0000-0003-3307-4148}},
\author[A]{\fnms{S.A.}~\snm{Sisson}\ead[label=e2]{Scott.Sisson@unsw.edu.au}\orcid{0000-0001-8943-067X}},
\and
\author[B]{\fnms{A.H.}~\snm{Welsh}\ead[label=e3]{Alan.Welsh@anu.edu.au}\orcid{0000-0002-3165-9559}}
\address[A]{UNSW Data Science Hub, and School of Mathematics and Statistics,
University of New South Wales\printead[presep={,\ }]{e1,e2}}

\address[B]{Research School of Finance, Actuarial Studies and Statistics,
Australian National University\printead[presep={,\ }]{e3}}
\end{aug}

\begin{abstract}
This paper presents asymptotic results for the maximum likelihood and restricted maximum likelihood (REML) estimators within a two-way crossed mixed effect model as the sizes of the rows, columns, and cells tend to infinity.  Under very mild conditions which do not require the assumption of normality, the estimators are proven to be asymptotically normal, possessing a structured covariance matrix. The growth rate for the number of rows, columns, and cells is unrestricted, whether considered pairwise or collectively.
\end{abstract}



\end{frontmatter}

\section{Introduction}

Mixed models are a diverse class of statistical models that relate a response variable to both observed covariates and latent random effects. They are widely used in applied statistics, particularly for modeling relationships within clustered and longitudinal data. Pioneering works on this topic include those by \cite{scott1982effect} and \cite{battese1988error} for survey data, and \cite{laird1982random} for longitudinal data. Commonly, linear mixed models are fitted by assuming normality and using either maximum likelihood (ML) or restricted maximum likelihood (REML) estimation; see for example \cite{harville1977maximum}. As these estimators are nonlinear, asymptotic methods are needed to derive their properties and construct approximate inferences for the unknown parameters.

The traditional asymptotic approach, discussed by  \cite{hartley1967maximum}, \cite{anderson1969statistical}, \cite{miller1977asymptotic}, \cite{das1979asymptotic}, \cite{cressie1993asymptotic}, and \cite{richardson1994asymptotic}, applies to data with independent clusters modelled by a nested random effect structure and requires the number of clusters to increase while the cluster size stays bounded.
Recent asymptotic results for mixed model estimators \citep{lyu2019estimation, jiang2022usable} allow both the number of clusters and the cluster sizes to increase, but still require nested random effects.  \cite{jiang2022usable} also require  the ratios of the cluster sizes to the  number of clusters to tend to zero. \cite{lyu2019estimation} do not impose such rate conditions, but their results are restricted to the nested error regression model.

There are real-world scenarios requiring crossed random effects that are not covered by the available results.  For example, an educational dataset recording student performance as the outcome variable can have teachers and class groups of students (treated as random effects) nested within schools but crossed with each other if teachers teach multiple class groups and/or class groups are taught by multiple teachers. In such applications, the clusters or cells (as clusters created by crossing are sometimes described) can be large, allowing the number of observations in them to be treated as asymptotically increasing.   In addition, there are theoretical problems (e.g.~in prediction, see \cite{jiang1998asymptotic}) for which consistency requires both the number of clusters and the cluster sizes to increase.  Moreover,
\citep[section 3.31] {jiang2017asymptotic} point out  that it is very difficult to derive the asymptotic behavior of the MLE in any generalized linear mixed effect model (GLMM) with increasing numbers of crossed random effects, even for the simplest cases. 
Although \cite{jiang2013subset} use a subset argument to prove consistency of the MLE for GLMM parameters with crossed random effects, the method does not lead to the asymptotic distribution of the estimators.
%
Partly for this reason, recent literature on GLMMs with crossed random effects focuses on computational algorithms such as \cite{ghosh2021scalable, ghosh2022backfitting, bellio2023scalable, menictas2023streamlined}; see for example \cite{ghosh2022backfitting} who proposed a backfitting algorithm to compute generalized least squares estimates for the crossed effect model without the interaction random effect, and \cite{menictas2023streamlined} who explored streamlined mean field variational Bayes algorithms for fitting linear mixed models with crossed random effects. 
There is a gap in the literature on asymptotic results for linear mixed models with crossed random effects.
The purpose of this paper is to develop rigorously the asymptotic properties of working normal maximum likelihood and REML estimators for the parameters of crossed random effect models when both the number of cells and their sizes tend to infinity.

We consider the two-way crossed random effect with interaction model as a canonical crossed random effect model.  It is used to analyze data structured by the crossing of two categorical factors, hereinafter referred to as factor A (rows) and factor B (columns), respectively.   
%
Let \( [y_{ijk}, \bx_{ijk}^T]^T \) represent the \( k \)th observed vector in cell \( (i,j) \), where \( y_{ijk} \) is a scalar response variable and \( \bx_{ijk} \) is a $p_0$-dimensional vector of covariates or explanatory variables, where $T$ denotes transpose. Indices \( i \) and \( j \) correspond to the row and column of the cell, respectively, and \( k \) identifies the specific vector within that cell. 
Let the total number of rows, columns, and observations per cell be \( g \), \( h \), and \( m \), respectively, so the sample size \( n = ghm \). (We explain in the next paragraph why we treat the balanced case with the same number of observations $m$ in each cell.)
The two-way crossed random effect with interaction model is
\begin{equation}\label{two-way cross model}
	y_{ijk}=\xi_0+\bx_{ijk}^T\bxi_s+\alpha_i+\beta_j+\gamma_{ij}+e_{ijk},
\end{equation}
for $i=1,\ldots,g, j=1,\ldots,h, k=1,\ldots,m,$
where $\xi_0$ is the intercept, $\bxi_s$ is the $p_0$-dimensional vector of  slope parameters, $\alpha_i$ is the random effect due to the $i$th row, $\beta_j$ is the random effect  due to the $j$th column, $\gamma_{ij}$ is the interaction random effect  of row $i$ and column $j$, and $e_{ijk}$ is the error term.  The variables $\alpha_i$, $\beta_j$, $\gamma_{ij}$, and $e_{ijk}$ are assumed mutually independent with zero means and variances $\sigma_\alpha^2$, $\sigma_\beta^2$, $\sigma_\gamma^2$, and $\sigma_e^2$, respectively; we do not assume normality.

Let $\by=[y_{111},\ldots, y_{ghm}]^T$,  $\bX^*=[\bx_{111},\ldots,\bx_{ghm}]^T$, $\balpha=[\alpha_1,\ldots,\alpha_g]^T$, $\bbeta=[\beta_1,$\ldots$,\beta_h]^T$, $\bgamma=[\gamma_{11},\ldots,\gamma_{gh}]^T$, and $\be=[e_{111},\ldots,e_{ghm}]^T$. 
The indices cycle through their possible values starting with the rightmost index first.
Then we can write (\ref{two-way cross model}) as 
\begin{equation}\label{two-way matrix}
	\by=\xi_0\bone_{n}+\bX^*\bxi_s+\bZ_1\balpha+\bZ_2\bbeta+\bZ_3\bgamma+\bZ_0\be,
\end{equation}
where $\bZ_0=\bI_g\otimes\bI_h\otimes \bI_m$,		$\bZ_1=\bI_g\otimes\bone_h\otimes \bone_m$, $\bZ_2=\bone_g\otimes\bI_h\otimes \bone_m$ and	$\bZ_3=\bI_g\otimes\bI_h\otimes \bone_m$, with $\bone_a$ the $a$-vector of ones, $\bI_a$ a $a \times a$ identity matrix, and $\otimes$ representing the Kronecker product.
The dispersion (variance-covariance) matrix of $\by$, denoted by $\bV$, is given by
\begin{equation}\label{Vmatrix}
	\bV=\bZ_1\bZ_1^T\sigma_\alpha^2+\bZ_2\bZ_2^T\sigma_\beta^2+\bZ_3\bZ_3^T\sigma_\gamma^2+\bZ_0\bZ_0^T\sigma_e^2,
\end{equation}
where
$\bZ_0\bZ_0^T=\bI_g\otimes\bI_h\otimes \bI_m$, $\bZ_1\bZ_1^T=\bI_g\otimes\bJ_h\otimes \bJ_m$,  $\bZ_2\bZ_2^T=\bJ_g\otimes\bI_h\otimes \bJ_m$ and 
$\bZ_3\bZ_3^T=\bI_g\otimes\bI_h\otimes \bJ_m$, with $\bJ_a = \bone_a\bone_a^T$, a $a \times a$ matrix of ones.
Under working normality for both the random effects and errors, the likelihood and the REML criterion can be written down as a function of $\bV^{-1}$, the inverse of the dispersion matrix (\ref{Vmatrix}); see Section \ref{Sec:2}. We call these functions the likelihood and the REML criterion, and refer to the parameter values that maximize them as the maximum likelihood  and REML estimators, respectively.  To simplify the study of these functions and the associated estimators, we use an exact expression for \(\bV^{-1}\) given by \cite{searle1979dispersion}.  This expression leverages the properties of the Kronecker product and is only available in the balanced case. Consequently, in this study, we focus only on balanced data when considering model (\ref{two-way cross model}).  We believe that our results are indicative of what occurs in the unbalanced case.

For our results, we do not assume normality, but instead assume the existence of finite "$4+\delta$" moments for both the random effects and errors. Whether the explanatory variables are considered fixed or random, we impose conditions akin to finite moments of order "$2+\delta$". These conditions are analogous to those used by \cite{ lyu2019eblup,lyu2019estimation} for nested error regression models. We permit the values of $g$, $h$, and $m$ to approach infinity, without restrictions on their relative rates of divergence.  The maximum likelihood and REML estimators converge at potentially four different rates, depending on the level of the parameter being estimated.  Analogously to \cite{lyu2019estimation}, we identify four distinct levels of parameters:
\begin{enumerate}
	\item {Factor A (row) parameters}: Coefficients of factor A (row) level covariates (covariates that vary only with $i$) and the corresponding variance $\sigma_{\alpha}^2$.
	\item {Factor B (column) parameters}: Coefficients of  factor B (column) level covariates (covariates that vary only with $j$) and the corresponding variance $\sigma_{\beta}^2$.	
	\item {Interaction parameters}: Coefficients of  interaction covariates (covariates that vary only with $i$ and $j$) and the corresponding variance $\sigma_{\gamma}^2$.	
	\item {Within cell parameters}: Coefficients of within  cell covariates (covariates that vary with $i$, $j$ and $k$) and the corresponding variance $\sigma_{e}^2$.
\end{enumerate}
The information for factor A parameters increases with \(g\), factor B parameters with \(h\), interaction parameters with \(gh\), and within-cell parameters  with \(n\). Each of these rates requires a different normalization on the corresponding estimators; we achieve this using a diagonal matrix which is readily interpretable.  We derive asymptotic representations for both the maximum likelihood and REML estimators in terms of their influence functions that are useful in their own right and for the development of the central limit theorem.  The asymptotic variance matrix we obtain is elegantly structured with a block diagonal configuration, which is both concise and interpretable, and simplifies making asymptotic inferences for the unknown parameters. 

The remainder of the paper is structured as follows. Section~\ref{Sec:2} introduces useful notation and presents an exact expression for the likelihood estimating equation. Section~\ref{sec:3} presents our main asymptotic theorem. Section~\ref{Sec:4} explains how to interpret and use our asymptotic theorem, including how to construct asymptotically valid statistical inferences for the model parameters. Section~\ref{sec:5} presents a simulation study to support the asymptotic theorem. Section~\ref{Proof} outlines lemmas and proofs for the main asymptotic theorem. Finally, Section~\ref{sec:7} concludes the paper and discusses potential future research directions.

\section{Notation}\label{Sec:2}
A general covariate $x_{ijk}$ can vary with all subscripts (a within cell covariate), with both $i$ and $j$ (an interaction covariate), or with $i$ (a row covariate) or $j$ (a column covariate) alone. We write $\bar{x}= n^{-1}\sum_{i=1}^g\sum_{j=1}^h \sum_{k=1}^m x_{ijk}$, $\bar{x}_{ij} = m^{-1}\sum_{k=1}^{m} x_{ijk}$, $\bar{x}_{i.} = (hm)^{-1} \sum_{j=1}^{h}\sum_{k=1}^{m} x_{ijk}$,  and $\bar{x}_{.j} = (gm)^{-1} \sum_{i=1}^{g}\sum_{k=1}^{m} x_{ijk}$, using the convention that we retain a dot as a placeholder subscript when required (as in $\bar{x}_{i.}$ and $\bar{x}_{.j}$) but not when there is no ambiguity (as in $\bar{x}_{ij}$ and $\bar{x}$).  Then we can decompose a within cell covariate $x_{ijk}$ as $\bar{x} + (\xidd-\xddd)+(\xdjd-\xddd)+(\xijd-\xidd-\xdjd+\xddd)+(x_{ijk}-\xijd)$, the sum of an overall mean, row, column, interaction, and within cell term.  Similarly, we can decompose a covariate at any level into components in each of the higher levels of the hierarchy.  This is important to accommodate the fact that the estimated coefficients of the covariates at each level have different rates of convergence; it also exactly orthogonalizes the covariates.  Finally, allowing the different terms in the decomposition to have possibly different coefficients in the model increases the flexibility of the model.
Following \citep{lyu2019estimation}, we arrange the covariate vector $\bx_{ijk}$ in four subvectors: a $p_a$-vector of row  covariates $\bxai$, a $p_b$-vector of column covariates $\bxbj$, a $p_{ab}$-vector of interaction covariates $\bxabij$, and a $p_w$-vector of within cell covariates $\bxwijk$, where $p=p_a+p_b+p_{ab}+p_w \ge p_0$. Not all of these vectors need be present in the model, but it is helpful to work with the most complete model.

The two-way crossed classification with interaction model (\ref{two-way cross model}) can be generalized to 
\begin{equation}\label{two-way crossed model-rw }
	y_{ijk}=\xi_0+\bxait\bxi_1+\bxbjt\bxi_2+\bxabijt\bxi_3+\bxwijkt\bxi_4+\alpha_i+\beta_j+\gamma_{ij}+e_{ijk},
\end{equation}
for $i=1,\ldots,g, j=1,\ldots,h, k=1,\ldots,m$, where $\xi_0$ is the  intercept, $\bxi_1$ is the $p_a$-vector of 
row slopes , $\bxi_2$  is the $p_b$-vector of column slopes, $\bxi_3$  is the $p_{ab}$-vector of interaction slopes, and $\bxi_4$  is the $p_w$-vector of within-cell slopes.
We denote a general parameter vector for the model (\ref{two-way crossed model-rw }) by $\bomega=[\xi_0,\bxi_1^T,\sigma_\alpha^2, \bxi_2^T,\sigma_\beta^2,\bxi_3^T,\sigma_\gamma^2, \bxi_4^T,\sigma_e^2]^T$ and the true parameter vector by  $\dot\bomega=[\dot\xi_0,\dot\bxi_1^T,\dot\sigma_\alpha^2, \dot\bxi_2^T,\dot\sigma_\beta^2,\dot\bxi_3^T,\dot\sigma_\gamma^2, \dot\bxi_4^T,\dot\sigma_e^2]^T$.  All expectations are taken under the true model with the covariates treated as fixed (i.e.~we condition on any random covariates).
To express model (\ref{two-way crossed model-rw }) in matrix form like (\ref{two-way matrix}),
let $\bxi=[\xi_0,\bxi_1^T,\bxi_2^T,\bxi_3^T,\bxi_4^T]^T$, 
$\bXa=[\bx_{1}^{(a)},\ldots,\bx_{g}^{(a)}]^T$,  $\bXb=[\bx_{1}^{(b)},\ldots,\bx_{h}^{(b)}]^T$,  $\bXab=[\bx_{11}^{(ab)},\ldots,\bx_{gh}^{(ab)}]^T$ and $\bXw=[\bx_{111}^{(w)},\ldots,\bx_{ghm}^{(w)}]^T$. Then the matrix form of  (\ref{two-way crossed model-rw }) is  
\begin{equation}\label{two-way matrix-rw}
	\by=\bX\bxi+\bZ_1\balpha+\bZ_2\bbeta+\bZ_3\bgamma+\bZ_0\be,
\end{equation}
where
$\bX=[\bone_{n},\bXa\otimes\bone_h\otimes\bone_m, \bone_g\otimes\bXb\otimes\bone_m,\bXab\otimes\bone_m,\bXw]$.

\newcommand{\cxia}{\bx_{i(c)}^{(a)}}
\newcommand{\cxiat}{\bx_{i(c)}^{(a)T}}
\newcommand{\cxiab}{\bar{\bx}_{i.(c)}^{(ab)}}
\newcommand{\cxiabt}{\bar{\bx}_{i.(c)}^{(ab)T}}
\newcommand{\cxiw}{\bar{\bx}_{i.(c)}^{(w)}}
\newcommand{\cxiwt}{\bar{\bx}_{i.(c)}^{(w)T}}
\newcommand{\cyia}{\bar{y}_{i.(c)}}

\newcommand{\cxjb}{\bx_{j(c)}^{(b)}}
\newcommand{\cxjbt}{\bx_{j(c)}^{(b)T}}
\newcommand{\cxjab}{\bar{\bx}_{.j(c)}^{(ab)}}
\newcommand{\cxjabt}{\bar{\bx}_{.j(c)}^{(ab)T}}
\newcommand{\cxjw}{\bar{\bx}_{.j(c)}^{(w)}}
\newcommand{\cxjwt}{\bar{\bx}_{.j(c)}^{(w)T}}
\newcommand{\cyjb}{\bar{y}_{.j(c)}}

\newcommand{\cxijab}{\bx_{ij(c)}^{(ab)}}
\newcommand{\cxijabt}{\bx_{ij(c)}^{(ab)T}}
\newcommand{\cxijw}{\bar{\bx}_{ij(c)}^{(w)}}
\newcommand{\cxijwt}{\bar{\bx}_{ij(c)}^{(w)T}}
\newcommand{\cyijab}{\bar{y}_{ij(c)}}

\newcommand{\cxijkw}{\bx_{ijk(c)}^{(w)}}
\newcommand{\cxijkwt}{\bx_{ijk(c)}^{(w)T}}
\newcommand{\cyijkw}{y_{ijk(c)}}

To simplify notation, we let $\eta=g/h$ and $\tau=\sigma_{\alpha}^2+\eta\sigma_{\beta}^2$ with
true value  $\dot\tau=\dot\sigma_{\alpha}^2+\eta\dot\sigma_{\beta}^2$.
We define the means of $y_{ijk}$ (and later $e_{ijk}$) analogously to the means of $x_{ijk}$.
For variables that are centered, we use the additional subscript $(c)$. We use superscripts $(a)$, $(b)$, $(ab)$, and $(w)$ to indicate that a variable is a row, column, interaction, or within-cell variable, respectively. 
For example, 
$ \mathbf{x}_{i(c)}^{(a)} = \mathbf{x}_{i}^{(a)}-\bar{\mathbf{x}}^{(a)}$, 
$ \bar{\mathbf{x}}_{i.(c)}^{(ab)} = \bar{\mathbf{x}}_{i.}^{(ab)}-\bar{\mathbf{x}}^{(ab)}$,  
$ \bar{\mathbf{x}}_{i.(c)}^{(w)} = \bar{\mathbf{x}}_{i.}^{(w)}-\bar{\mathbf{x}}^{(w)}$, 
and similarly for terms such as 
$ \mathbf{x}_{j(c)}^{(b)}$, 
$ \bar{\mathbf{x}}_{.j(c)}^{(ab)}$, 
$ \bar{\mathbf{x}}_{.j(c)}^{(w)}$, 
$ \bar{\mathbf{x}}_{ij(c)}^{(ab)}$, 
$ \bar{\mathbf{x}}_{ij(c)}^{(w)}$, 
and 
$ \mathbf{x}_{ijk(c)}^{(w)}$.
We also define sums of squares for certain variables. 
Let $\bX_{i(c)}^{(a)}=[\cxiat,\cxiabt,\cxiwt]^T$,  $\bX_{j(c)}^{(b)}=[\cxjbt,\cxjabt,\cxjwt]^T$ and  $\bX_{ij(c)}^{(ab)}=[\cxijabt,\cxijwt]^T$. Then we write
\begin{align*}
	&\sumig\bX_{i(c)}^{(a)}\bX_{i(c)}^{(a)T}=\begin{bmatrix}
		SA_{(a)}^{x}& 	SA_{(a),(ab)}^{x}& 			SA_{(a),(w)}^{x}\\
		SA_{(ab),(a)}^{x}& 	SA_{(ab)}^{x}& 			SA_{(ab),(w)}^{x}\\
		SA_{(w),(a)}^{x}& 	SA_{(w),(ab)}^{x}& 	 SA_{(w)}^{x}		
	\end{bmatrix},\\&
 		\sumjh\bX_{j(c)}^{(b)}\bX_{j(c)}^{(b)T}=\begin{bmatrix}
		SB_{(b)}^{x}& 	SB_{(b),(ab)}^{x}& 			SB_{(b),(w)}^{x}\\
		SB_{(ab),(b)}^{x}& 	SB_{(ab)}^{x}& 			SB_{(ab),(w)}^{x}\\
		SB_{(w),(b)}^{x}& 	SB_{(w),(ab)}^{x}& 	 SB_{(w)}^{x}		
	\end{bmatrix},\\
 &
	\sumig\sumjh \bX_{ij(c)}^{(ab)}\bX_{ij(c)}^{(ab)T}=\begin{bmatrix}
		SAB_{(ab)}^{x}& SAB_{(ab),(w)}^{x}\\
		SAB_{(w),(ab)}^{x}& SAB_{(w)}^{x}
	\end{bmatrix},\\ &
 \quad\text{and}
\qquad
	\sumig\sumjh\sumkm\cxijkw\cxijkwt=SW_{(w)}^{x}.
\end{align*}
We define the block matrices $\hat\bD_{1}=g^{-1}SA_{(a)}^{x}$, $\hat\bD_2=h^{-1}SB_{(b)}^{x}$, $\hat\bD_3= (gh)^{-1}SAB_{(ab)}^{x}$ and $\hat\bD_4=n^{-1}SW_{(w)}^{x}$,  and their respective limits  $\bD_{1}=\lim_{g\to\infty}\hat\bD_{1}$, $\bD_2=\lim_{h\to\infty}\hat\bD_{2}$, $\bD_3=\lim_{g,h\to\infty}\hat\bD_{3}$ and $\bD_4=\lim_{g,h,m\to\infty}\hat\bD_{4}$.
For a comprehensive list of all notation, refer to Appendix~\ref{appA}.

	%
	%
	%





	The log-likelihood for the parameters in model (\ref {two-way matrix-rw}) is
	\begin{equation}\label{lglik}
		l(\bomega;\by)=-\frac{n}{2}\log 2\pi -\frac{1}{2}\log|\bV|-\frac{1}{2}(\by-\bX\bxi)^T\bV^{-1}(\by-\bX\bxi),
	\end{equation}
	where $|\bV|$ is the determinant of $\bV$.
	To write down the inverse of $\bV$, we define $\bJ_a^0=\bI_a$ and $\bJ_a^1=\bJ_a$, so the dispersion matrix (\ref{Vmatrix}) is written as 
	\begin{equation}\label{Vmatrix binary}
		\bV=\sigma_\alpha^2(\bJ_g^0\otimes\bJ_h^1\otimes\bJ_m^1)+\sigma_\beta^2(\bJ_g^1\otimes\bJ_h^0\otimes\bJ_m^1)+\sigma_\gamma^2(\bJ_g^0\otimes\bJ_h^0\otimes\bJ_m^1)+\sigma_e^2(\bJ_g^0\otimes\bJ_h^0\otimes\bJ_m^0).
	\end{equation}
	Following  \citep{searle1979dispersion}, we obtain the explicit expression
	\begin{align*}
		\bV^{-1}&=
		\frac{1}{\lambda_0}(\bI_g\otimes\bI_h\otimes\bC_m)+\frac{1}{\lambda_1}(\bC_g\otimes\bC_h\otimes\bar{\bJ}_m)+\frac{1}{\lambda_2}(\bC_g\otimes\bar{\bJ}_h\otimes\bar{\bJ}_m)\\&
		\qquad+\frac{1}{\lambda_3}(\bar{\bJ}_g\otimes\bC_h\otimes\bar{\bJ}_m)+\frac{1}{\lambda_4}(\bar{\bJ}_g\otimes\bar{\bJ}_h\otimes\bar{\bJ}_m),
	\end{align*}
	where 
$\lambda_0=\sigma_e^2$, $\lambda_1=\sigma_e^2+m\sigma_\gamma^2$,
$\lambda_2=\sigma_e^2+m\sigma_\gamma^2+hm\sigma_\alpha^2$,
$\lambda_3=\sigma_e^2+m\sigma_\gamma^2+gm\sigma_\beta^2$, 
$\lambda_4=\sigma_e^2+m\sigma_\gamma^2+hm\sigma_\alpha^2+gm\sigma_\beta^2$, $\bar{\bJ}_a=\frac{1}{a}{\bJ}_a$ and $\bC_a=\bI_a-\bar\bJ_a$. The details are given in Appendix~\ref{AppendixB}.

To find the maximum likelihood estimator $\hbomega$ of $\bomega$, we differentiate (\ref{lglik}) with respect to $\bomega$ to obtain the estimating function $\bpsi(\bomega)$ and then solve the estimating equation $\bzero_{[(p+5):1]}=\bpsi(\bomega)$, where $\bzero_{[p:q]}$ is a $p\times q$ matrix of zeros.  The derivation of the estimating function $\boldsymbol{\psi}(\boldsymbol{\omega})$ is given in Supplementary Section S.2.
Henceforth, to be concise and maintain clarity, the argument $\bomega$ will be excluded from expressions when  doing so does not affect the intended meaning or introduce ambiguity.
We define  $\bar{R}=\yddd-\xi_0-\bar{\bx}^{(a)T}\bxi_1-\bar{\bx}^{(b)T}\bxi_2-\bar{\bx}^{(ab)T}\bxi_3-\bar{\bx}^{(w)T}\bxi_4$,  $\barRia=\cyia-\cxiat\bxi_1-\cxiabt\bxi_3-\cxiwt\bxi_4$, $\barRjb=\cyjb-\cxjbt\bxi_2-\cxjabt\bxi_3-\cxjwt\bxi_4$, $\barRijab=\cyijab-\cxijabt\bxi_3-\cxijwt\bxi_4$ and $\barRijkw=\cyijkw-\cxijkwt\bxi_4$.
Then the components of $\bpsi(\bomega)$ are
	\begin{equation}\label{mlequation}
		\begin{split}
			l_{\xi_0}&=\frac{n}{\lambda_4}\bar{R},
			\qquad
			l_{\bxi_1}=\frac{hm}{\lambda_2}\sumig\cxia\barRia+\frac{n}{\lambda_4}\bar{\bx}^{(a)}\bar{R},\\
			l_{\sigma_\alpha^2}&=\frac{h^2m^2}{2\lambda_2^2}\sumig\barRia^2-\frac{(g-1)hm}{2\lambda_2}+\frac{gh^2m^2}{2\lambda_4^2}\bar{R}^2-\frac{hm}{2\lambda_4},\qquad
			l_{\bxi_2}=\frac{gm}{\lambda_3}\sumjh\cxjb\barRjb+\frac{n}{\lambda_4}\bar{\bx}^{(b)}\bar{R},\\
			l_{\sigma_\beta^2}&=\frac{g^2m^2}{2\lambda_3^2}\sumjh\barRjb^2-\frac{(h-1)gm}{2\lambda_3}+\frac{g^2hm^2}{2\lambda_4^2}\bar{R}^2-\frac{gm}{2\lambda_4},\\
			l_{\bxi_3}&= \frac{m}{\lambda_1}\sumig\sumjh\cxijab\barRijab
			+\frac{hm}{\lambda_2}\sumig\cxiab\barRia
			+\frac{gm}{\lambda_3}\sumjh\cxjab\barRjb
			+\frac{n}{\lambda_4}\barbxab\bar{R},\\
			l_{\sigma_\gamma^2}&=\frac{m^2}{2\lambda_1^2}\sumig\sumjh\barRijab^2-\frac{(g-1)(h-1)m}{2\lambda_1}+\frac{hm^2}{2\lambda_2^2}\sumig\barRia^2 -\frac{(g-1)m}{2\lambda_2}+\frac{gm^2}{2\lambda_3^2}\sumjh\barRjb^2\\&\qquad
			-\frac{(h-1)m}{2\lambda_3} +\frac{ghm^2}{2\lambda_4^2}\bar{R}^2-\frac{m}{2\lambda_4},\\
			l_{\bxi_4}&=  	\frac{1}{\lambda_0}\sumig\sumjh\sumkm\cxijkw\barRijkw
			+\frac{m}{\lambda_1}\sumig\sumjh\cxijw\barRijab
			+\frac{hm}{\lambda_2}\sumig\cxiw\barRia
			+\frac{gm}{\lambda_3}\sumjh\cxjw\barRjb
			+\frac{n}{\lambda_4}\barbxw\bar{R},\\
			l_{\sigma_e^2}&=\frac{1}{2\lambda_0^2}\sumig\sumjh\sumkm\barRijkw^2
			-\frac{gh(m-1)}{2\lambda_0}	+\frac{m}{2\lambda_1^2}\sumig\sumjh\barRijab^2-\frac{(g-1)(h-1)}{2\lambda_1}
		\\&\qquad
			+\frac{hm}{2\lambda_2^2}\sumig\barRia^2	-\frac{g-1}{2\lambda_2}
			+\frac{gm}{2\lambda_3^2}\sumjh\barRjb^2-\frac{h-1}{2\lambda_3}
			+\frac{n}{2\lambda_4^2}\bar{R}^2-\frac{1}{2\lambda_4}.
		\end{split}
	\end{equation}

	We write $\bpsi(\bomega)=[	\bpsi^{(a)}(\bomega)^T, \bpsi^{(b)}(\bomega)^T, \bpsi^{(ab)}(\bomega)^T, \bpsi^{(w)}(\bomega)]^T$,  
	where $\bpsi^{(a)}(\bomega)= [ l_{\xi_0}, l_{\bxi_1}^T, l_{\sigma_\alpha^2}]^T$ are the estimating functions for the row parameters, 
	$	\bpsi^{(b)}(\bomega)=[l_{\bxi_2}^T, l_{\sigma_\beta^2}]^T$ are the estimating functions for the column  parameters, 
	$ \bpsi^{(ab)}(\bomega)=[l_{\bxi_3}^T, l_{\sigma_\gamma^2} ]^T$ are the estimating functions for the interaction parameters
	and 
	$ 	\bpsi^{(w)}(\bomega)= [l_{\bxi_4}^T, l_{\sigma_e^2}]^T$ are the estimating functions for the within-cell parameters. The derivative of the estimating function,  $\triangledown \bpsi (\bomega)$, and its expected value under the model is given in Supplementary Section S.5.

	Throughout the paper, we let $|\ba|=(\ba^T\ba)^{1/2}$ and $\|\bA\|=\{\text{trace}(\bA\bA^T)\}^{1/2}$ denote the Euclidean (Frobenius) norm of the vector $\ba$ and the matrix $\bA$, respectively.

	\section{Asymptotic Normality Theorem}\label{sec:3}
	To approximate the estimating function and derive the asymptotic properties of $\hat\bomega$ from the estimating equation, we impose the following condition.

 \medskip
	\noindent
	\textbf{Condition A}
	\begin{enumerate}
		\item  The model (\ref{two-way crossed model-rw })  holds with true parameter $\dot\bomega$ inside the parameter space $\bOmega$.
		\item The number of factor A levels (rows) $g\to\infty$, the number of factor B levels (columns) $h\to\infty$ and the number of observations within each cell  $m\to\infty$.
		\item The random variables $\alpha_i$, $\beta_j$, $\gamma_{ij}$ and $e_{ijk}$ are independent and identically distributed and mutually independent.  Moreover, there is a $\delta>0$, such that $\E |\alpha_1|^{4+\delta}<\infty$, $\E |\beta_1|^{4+\delta}<\infty$, $\E |\gamma_{11}|^{4+\delta}<\infty$ and $\E| e_{111}|^{4+\delta}<\infty$.
		\item The limits $\lim_{g\to\infty} g^{-1} \sumig \bxai = \barbxa$ and $\lim_{h\to\infty} \sumjh \bxbj = \barbxb$ exist. Also, the limits of the matrices $\lim_{g\to\infty}g^{-1}\sumig\bX_{i(c)}^{(a)}\bX_{i(c)}^{(a)T}$, \\
  $\lim_{h\to\infty}h^{-1}\sumjh\bX_{j(c)}^{(b)}\bX_{j(c)}^{(b)T}$, $\lim_{g,h\to\infty} (gh)^{-1}\sumig\sumjh\bX_{ij(c)}^{(ab)}\bX_{ij(c)}^{(ab)T}$ and \\
  $\lim_{g,h,m\to\infty} n^{-1}\sumig\sumjh\sumkm\cxijkw\cxijkwt$ exist and are positive definite. Additionally,  there is a $\delta>0$, such that  $\lim_{g\to\infty}g^{-1}\sumig|\cxia|^{2+\delta}<\infty$, \\
  $\lim_{h\to\infty}h^{-1}\sumjh |\cxjb|^{2+\delta}<\infty$, $\lim_{g,h\to\infty}(gh)^{-1}\sumig\sumjh|\cxijab|^{2+\delta}<\infty$ and $ \lim_{g,h,m\to\infty}n^{-1}\sumig\sumjh\sumkm|\cxijkw|^{2+\delta}<\infty$.
	\end{enumerate}
	As noted in the Introduction, these are mild conditions.  
	Condition A2 allows $g, h, m$ to diverge to infinity without imposing constraints on the relative rates. 
	Conditions A3 and A4 ensure that limits needed for the existence of the asymptotic variance of the estimating function exist and that we can establish a Lyapunov condition and hence a central limit theorem for the estimating function. They also ensure that the negative of the appropriately normalized second derivative of the estimating function converges to the matrix $\bB$ given in (\ref{MatrixB}). 
	
	Our main result is the following theorem which we prove in Section~\ref{Proof}.
	\begin{theoremm}\label{theorem1}
		Suppose Condition A holds. 	
		When  $0\le \eta<\infty$, there is a solution $\hat\bomega$ to the estimating equations
		$\bpsi(\bomega)=\bzero_{[(p+5):1]}$, satisfying $|\bK^{1/2}(\hat\bomega-\dot\bomega)|=O_p(1)$, where $\bK=\diag(g\bI_{p_a+2},h\bI_{p_b+1},gh\bI_{p_{ab}+1},n\bI_{p_w+1})$ where $\bI_p$ is the $p\times p$ identity matrix. Moreover, $\hat\bomega$ has the asymptotic representation
		\begin{equation}
			\bK^{1/2}(\hat\bomega-\dot\bomega)=\bB^{-1}\bK^{-1/2}\bphi+o_p(1),
		\end{equation}
		where $\bB$ is given by (\ref{MatrixB}), and $\bphi=[\phi_{\xi_0},\bphi_{\bxi_1}^T,\phi_{\sigma_{\alpha}^2},\bphi_{\bxi_2}^T,\phi_{\sigma_{\beta}^2},\bphi_{\bxi_3}^T,\phi_{\sigma_{\gamma}^2},\bphi_{\bxi_4}^T,\phi_{\sigma_{e}^2}]^T$ has components
		\begin{gather}
			\begin{aligned}\label{eq:approx}
				\phi_{\xi_0}&=\frac{1}{\dot{\tau}}(\sumig \alpha_i+\eta\sumjh\beta_j),
				&\bphi_{\bxi_1}&=\frac{1}{\dot{\sigma}_\alpha^2}\sumig (\bxai-\frac{\eta \dot{\sigma}_\beta^2}{\dot{\tau}}\barbxa)\alpha_i+\frac{\eta}{\dot{\tau}}\barbxa\sumjh\beta_j,\\
				\phi_{\sigma_\alpha^2}&=\frac{1}{2\dot{\sigma}_\alpha^4}\sumig(\alpha_i^2-\dot{\sigma}_\alpha^2),
				&\bphi_{\bxi_2}&= \frac{1}{\dot{\sigma}_{\beta}^2}\sumjh(\bxbj-\frac{\dot{\sigma}_{\alpha}^2}{\dot{\tau}}\barbxb) \beta_j+\frac{1}{\dot{\tau}}\barbxb\sumig\alpha_i,\\
				\phi_{\sigma_\beta^2}&=\frac{1}{2\dot{\sigma}_\beta^4}\sumjh(\beta_j^2-\dot{\sigma}_\beta^2),
				&\bphi_{\bxi_3}&=\frac{1}{\dot{\sigma}_{\gamma}^2}\sumig\sumjh\cxijab\gamma_{ij},
				\qquad \phi_{\sigma_\gamma^2}=\frac{1}{2\dot{\sigma}_{\gamma}^4}\sumig\sumjh(\gamma_{ij}^2-\dot{\sigma}_\gamma^2),\\
				\bphi_{\bxi_4}&=\frac{1}{\dot{\sigma}_{e}^2}\sumig\sumjh\sumkm\cxijkw e_{ijk},
				&\phi_{\sigma_e^2}&=\frac{1}{2\dot{\sigma}_e^4}\sumig\sumjh\sumkm(e_{ijk}^2-\dot{\sigma}_{e}^2).		
			\end{aligned}
		\end{gather}
		It follows that 
		\begin{displaymath}
			\bK^{1/2}(\hat\bomega-\dot\bomega) \xrightarrow{D}  N(\bzero, \bF),
			\end{displaymath}	 
			where $\bF=\diag\left(\begin{bmatrix}
				\bF_{1(a),(a)}&\bF_{1(a),(b)}\\	\bF_{1(b),(a)}&\bF_{1(b),(b)}
			\end{bmatrix},\bF_2,\bF_3\right)$, 
			with 	 		
			\begin{align*}
				&\bF_{1(a),(a)}=\begin{bmatrix}
					f_1
					&-\dsiga^2\bbf_2&\E\alpha_1^3\\
					-\dsiga^2\bbf_2^T&\dsiga^2\bD_{1}^{-1}&\bzero_{[p_a:1]}\\
					\E\alpha_1^3&\bzero_{[1:p_a]}&\E\alpha_1^4-\dsiga^4
				\end{bmatrix},\quad
				\bF_{1(a),(b)}=	\bF_{1(b),(a)}^T=
				\begin{bmatrix}
					-\eta^{1/2}\dsigb^2\bbf_3&\eta^{1/2}\E\beta_1^3\\
					\bzero_{[p_a:p_b]}&\bzero_{[p_b:1]}\\
					\bzero_{[1:p_b]}&0
				\end{bmatrix},
				\\&	
				\bF_{1(b),(b)}=\begin{bmatrix}
					\dsigb^2\bD_{2}^{-1}&\bzero_{[p_b:1]}\\
					\bzero_{[1:p_b]}&\E\beta_1^4-\dsigb^4
				\end{bmatrix},\,\,\,
				\bF_2=\begin{bmatrix}
					{\dot{\sigma}_\gamma^2}{\bD_3^{-1}}&\bzero_{[p_{ab}:1]}\\
					\bzero_{[1:p_{ab}]}&{\E\gamma_{11}^4-\dot{\sigma}_\gamma^4}
				\end{bmatrix}
				\,\,\text{and}\,\,
				\bF_3=\begin{bmatrix}
					{\dot{\sigma}_e^2}{\bD_4^{-1}}&\bzero_{[p_w:1]}\\
					\bzero_{[1:p_w]}&{\E e_{111}^4-\dot{\sigma}_e^4} 
				\end{bmatrix},
			\end{align*}
			$f_1=	\dot\tau+\dsiga^2\barbxat\bD_{1}^{-1}\barbxa+\eta\dsigb^2\barbxbt\bD_{2}^{-1}\barbxb$,  $\bbf_2=\barbxat\bD_{1}^{-1}$ and  $\bbf_3=\barbxbt\bD_{2}^{-1}$.
   The result for $\eta = \infty$ can be obtained from that for $\eta=0$ by reordering the parameter vector as $\bomega = [\bxi_1^T,\sigma^2_a,\xi_0,\bxi_2^T,\sigma^2_\beta,\bxi_3^T,\sigma^2_{\gamma},\bxi_4^T,\sigma^2_e]^T$ and making the same replacements in the asymptotic coavriate matrix.
		\end{theoremm}
		\begin{corollary}\label{corollary}
			Suppose Condition A holds. 	If  $0\le \eta<\infty$, then we can estimate $\bF$ consistently using
			$\hat\bF=\diag\left(\begin{bmatrix}
				\hat\bF_{1(a),(a)}&\hat\bF_{1(a),(b)}\\	\hat\bF_{1(b),(a)}&\hat\bF_{1(b),(b)}
			\end{bmatrix},\hat\bF_2,\hat\bF_3\right)$, 
				where
				\begin{align*}
					&\hat\bF_{1(a),(a)}=\begin{bmatrix}
						\hat{f}_1
						&-\hsiga^2\hat{\bbf}_2&\hat\mu_{3\alpha}\\
						-\hsiga^2\hat{\bbf}_2^T&\hsiga^2\hat\bD_{1}^{-1}&\bzero_{[p_a:1]}\\
						\hat\mu_{3\alpha}&\bzero_{[1:p_a]}&\hat\mu_{4\alpha}-\hsiga^4
					\end{bmatrix},\qquad
					\hat\bF_{1(a),(b)}=	\hat\bF_{1(b),(a)}^T=
					\begin{bmatrix}
						-\eta^{1/2}\hsigb^2\hat{\bbf}_3&\eta^{1/2}\hat\mu_{3\beta}\\
						\bzero_{[p_a:p_b]}&\bzero_{[p_b:1]}\\
						\bzero_{[1:p_b]}&0
					\end{bmatrix},\\&
					\hat\bF_{1(b),(b)}=\begin{bmatrix}
						\hsigb^2\hat\bD_{2}^{-1}&\bzero_{[p_b:1]}\\
						\bzero_{[1:p_b]}&\hat\mu_{4\beta}-\hsigb^4
					\end{bmatrix},
					\,\,\,
					\hat\bF_2=\begin{bmatrix}
						{\hat{\sigma}_\gamma^2}{\hat\bD_3^{-1}}&\bzero_{[p_{ab}:1]}\\
						\bzero_{[1:p_{ab}]}&\hat\mu_{4\gamma}-\hat{\sigma}_\gamma^4
					\end{bmatrix}
					\,\,\text{and}\,\,
					\hat\bF_3=\begin{bmatrix}
						{\hat{\sigma}_e^2}{\hat\bD_4^{-1}}&\bzero_{[p_w:1]}\\
						\bzero_{[1:p_w]}&{\mu_{4e}-\hat{\sigma}_e^4} 
					\end{bmatrix},
				\end{align*} 
				with $\hat{f}_1=\hat\tau+\hsiga^2\barbxat\hat\bD_{1}^{-1}\barbxa+\eta\hsigb^2\barbxbt\hat\bD_{2}^{-1}\barbxb$, $\hat{\bbf}_2=\barbxat\hat\bD_{1}^{-1}$, $\hat{\bbf}_3=\barbxbt\hat\bD_{2}^{-1}$,
				$\hat\mu_{k\alpha}=g^{-1}\sumig(\bar{r}_{i.}-\bar{r})^k$, 
				$\hat\mu_{k\beta}=h^{-1}\sumjh(\bar{r}_{.j}-\bar{r})^k$, for $k=3,4$,
				$\hat{\mu}_{4\gamma}= (gh)^{-1}\sumig\sumjh (\bar{r}_{ij}-\bar{r}_{i.}-\bar{r}_{.j}+\bar{r})^4$,
				and  $\hat\mu_{4e}=n^{-1}\sumig\sumjh\sumkm(r_{ijk}-\bar{r}_{ij})^k$, 
				using $r_{ijk}=y_{ijk}-\hat\xi_0-\bxait\hat\bxi_1-\bxbjt\hat\bxi_2-\bxabijt\hat\bxi_3-\bxwijkt\hat\bxi_4$,
				$ \bar{r}_{ij} = m^{-1}\sum_{k=1}^{m} y_{ijk}$, 
				$ \bar{r}_{i.} = (hm)^{-1} \sum_{j=1}^{h}\sum_{k=1}^{m} y_{ijk}$,  
				$ \bar{r}_{.j} = (gm)^{-1} \sum_{i=1}^{g}\sum_{k=1}^{m} y_{ijk}$, 
				and
				$ \bar{r} = n^{-1} \sum_{i=1}^{g}\sum_{j=1}^{h}\sum_{k=1}^{m} y_{ijk}$,
				for $i=1,\ldots,g$, $j=1,\ldots,h$, $k=1\ldots,m$.
    The result for $\eta = \infty$ can be obtained from that for $\eta=0$ by using the same reordering of the parameter vector as in Theorem \ref{theorem1}.
			\end{corollary}
			We now make some remarks concerning Theorem 1 and Corollary 2.
			\begin{enumerate}
   \item 	The result for the case $\eta = \infty$ ($g$ tends to infinity faster than $h$) in Theorem \ref{theorem1} follows from the symmetry between the row and column parameters that allows us to swap them with each other.  The one sublety is that the intercept parameter $\xi_0$ is a row parameter when $\eta=0$ and a column parameter when $\eta=\infty$. 
				\item The case $\eta=0$ ($g$ tends to infinity slower than $h$) leads to a simplification of the above results. For example, $\dot{\tau}=\dsiga^2$, and three components of $\bphi$ can be simplified to $\phi_{\xi_0} = {1}/{\dsiga^2}\sumig \alpha_i$, $\bphi_{\bxi_1} = {1}/{\dot{\sigma}_\alpha^2}\sumig \bxai\alpha_i$ and $\bphi_{\bxi_2} = {1}/{\dot{\sigma}_{\beta}^2}\sumjh\bxbj\beta_j$.
				Additionally, the matrix $\bF$ can be simplified to $\bF=\diag\left( \bF_{1(a),(a)}, \bF_{1(b),(b)}, \bF_2, \bF_3 \right)$, where $\bF_{1(b),(b)}$, $\bF_2$ and $\bF_3$ remain unchanged, while $\bF_{1(a),(b)}=\bF_{1(b),(a)}^T=\bzero_{[(p_a+2):(p_b+1)]}$ and
				\begin{align*}
					\bF_{1(a),(a)}=\begin{bmatrix}
						\dsiga^2(1+\barbxat\bD_{1}^{-1}\barbxa)
						&-\dsiga^2\barbxat\bD_{1}^{-1}&\E\alpha_1^3\\
						-\dsiga^2\bD_{1}^{-1}\barbxa&\dsiga^2\bD_{1}^{-1}&\bzero_{[p_a:1]}\\
						\E\alpha_1^3&\bzero_{[1:p_a]}&\E\alpha_1^4-\dsiga^4
					\end{bmatrix}.
				\end{align*}
    An analogous simplification holds when $\eta=\infty$.
    
				\item The case \(0<\eta<\infty\) (i.e.~\(g\) and \(h\) tend to infinity at the same rate) is the most complex situation. The row parameters  (\( [\xi_0,\bxi_1^T,\siga^2]^T \)), are not asymptotically independent of the column parameters ( \( [\bxi_2^T,\sigb^2]^T \)). As \( g,h,m \to \infty \), we find that
				$\lim_{g,h,m\to\infty} {n}/{\dlambda_4} = {g}/{(\dsiga^2+\eta\dsigb^2)}$, where $\dlambda_4$ is the true value of $\lambda_4$.
				Consequently, the approximating functions for estimating the row parameters  \( \phi_{\xi_0} \) and \( \phi_{\bxi_1} \) contain both $\alpha_i$ (the row random effects) and \( \beta_j \) (the column random effects). Similarly, the  approximating functions for estimating the column parameter  \( \bphi_{\bxi_2} \) contain both row and column random effects.  The details are given in Supplementary Section S.3.
				In this case, we have to to consider the estimators of the row parameters and column parameters together: \( [\xi_0,\bxi_1^T,\siga^2,\bxi_2^T,\sigb^2]^T \). 

			\end{enumerate}
			
			For REML estimation, group the regression parameters as
			$\bxi=[\xi_0,\bxi_1^T,\bxi_2^T,\bxi_3^T,\bxi_4^T]^T$
			and the variance components as
			$\btheta=[\siga^2,\sigb^2,\siggama^2,\sige^2]^T$.
			The REML criterion function is obtained by replacing the regression parameters $\bxi$ in the log-likelihood (\ref{lglik}) by their maximum likelihood estimators for each $\btheta$ and then modifying the resulting profile log-likelihood for $\btheta$ to reduce the bias of the estimators. For each $\btheta$, the estimating equations for $\bxi$ presented in (\ref{mlequation}) are solved by
			$
			\hat\bxi(\btheta)=\{\bX^T\bV^{-1}\bX\}^{-1}\bX^T\bV^{-1}\by,
			$
			and the REML criterion function is
			\[
			l_R(\btheta;\by)=l(\hat\bxi(\btheta),\btheta;\by)-\frac{1}{2}\log|\bX^T\bV^{-1}\bX|.
			\]
			The REML estimator  $\hbtheta_R$ of $\dot{\btheta}$ maximizes the REML criterion function $l_R(\btheta;\by)$. We define $\hat\bxi_{R}=\hat{\bxi}(\hbtheta_R)$ to be the REML estimator of $\bxi$ and write the REML estimator of $\dot{\bomega}$ as
			$
			\hat\bomega_R=[\hat\xi_{R0},\hat\bxi_{R1}^T,\hat\sigma_{R\alpha}^2, \hat\bxi_{R2}^T,\hat\sigma_{R\beta}^2,\hat\bxi_{R3}^T,\hat\sigma_{R\gamma}^2, \hat\bxi_{R4}^T,\hat\sigma_{Re}^2]^T.
			$
			Since $\bX^T\bV^{-1}\bX$ is not a function of $\bxi$, the REML estimator also maximizes the adjusted log-likelihood
			\[
			l_A(\bxi,\btheta;\by)=l(\bomega;\by)-\frac{1}{2}\log|\bX^T\bV^{-1}\bX|.
			\]
			According to \citep{patefield1977maximized}, the REML estimator can be identified in a single step (rather than in two steps) by optimizing $l_A(\bxi,\btheta;\by)$ directly. This means that the estimating function is represented by $\bpsi_A(\bomega)$ using the derivatives $l_{A\xi_0}=l_{\xi_0}$,
			$l_{A\bxi_1}=l_{\bxi_1}$, $l_{A\bxi_2}=l_{\bxi_2}$, $l_{A\bxi_3}=l_{\bxi_3}$, $l_{A\bxi_4}=l_{\bxi_4}$, and
			\begin{align*}
				&l_{A\siga^2}(\bomega)=l_{\siga^2}(\bomega)-\frac{1}{2}\operatorname{trace}\{ (\bX^T\bV^{-1}\bX)^{-1}\bX^T\bV^{-1}\bZ_1\bZ_1^T\bV^{-1}\bX\},\\
				&l_{A\sigb^2}(\bomega)=l_{\sigb^2}(\bomega)-\frac{1}{2}\operatorname{trace}\{ (\bX^T\bV^{-1}\bX)^{-1}\bX^T\bV^{-1}\bZ_2\bZ_2^T\bV^{-1}\bX\},\\
				&l_{A\siggama^2}(\bomega)=l_{\siggama^2}(\bomega)-\frac{1}{2}\operatorname{trace}\{ (\bX^T\bV^{-1}\bX)^{-1}\bX^T\bV^{-1}\bZ_3\bZ_3^T\bV^{-1}\bX\},\\
				&l_{A\sige^2}(\bomega)=l_{\sige^2}(\bomega)-\frac{1}{2}\operatorname{trace}\{ (\bX^T\bV^{-1}\bX)^{-1}\bX^T\bV^{-1}\bZ_0\bZ_0^T\bV^{-1}\bX\}.
			\end{align*}
We show that the REML estimator is asymptotically equivalent to the maximum likelihood estimator by showing that the
contribution from the adjustment terms to the estimating function is asymptotically negligible. This yields the following
theorem which we prove in Section~\ref{Proof}.		
			\begin{theoremm}\label{theorem3}
				Suppose Condition A holds. 
    Then there is a solution $\hat{\bomega}_R$ of the REML estimating equations satisfying  $|\bK^{1/2}(\hat{\bomega}_R-\dot\bomega)|=O_p(1)$ and 
				\begin{align*}
					\bK^{1/2}(\hat{\bomega}_R-\hat\bomega)=o_p(1),
				\end{align*}
				so Theorem 1 applies to the REML estimator.
			\end{theoremm}

 			\section{Interpreting and Using the Results}\label{Sec:4}
			Theorems \ref{theorem1} and \ref{theorem3} establish asymptotic equivalence, asymptotic representations, and asymptotic normality for the maximum likelihood and REML estimators of the parameters in two-way crossed mixed effect  with interaction model under mild conditions when  the numbers of rows, columns, and observations within cells tend to infinity. In this section, we interpret and discuss these results.

			 We can estimate \(\dot\bomega\) consistently when both \(g\) and \(h\) approach infinity, given a bounded cell size \(m\). However, to consistently estimate the random effects \(\{\alpha_i,\beta_j,\gamma_{ij}\}\), \(m\) also needs to approach infinity \cite{jiang1998asymptotic}. Specifically:
			 
			 \begin{enumerate}
  			 	\item When \(m \to \infty\) with fixed \(g\) and \(h\), we can consistently estimate the within-cell variance \(\sigma_e^2\). However, estimates for the row variance \(\siga^2\), the  column variance \(\sigb^2\), and the interaction variance \(\sigma_{\gamma}^2\) will not be consistent.
			 	\item When both \(m\) and \(g\) approach infinity with \(h\) held fixed, we can consistently estimate the row variance \(\siga^2\), the interaction variance \(\sigma_{\gamma}^2\), and the within-cell variance \(\sigma_e^2\). However, the column variance \(\sigb^2\) will not be consistently estimated.
			 	\item When both \(m\) and \(h\) approach infinity with \(g\) held fixed, we can consistently estimate the column variance \(\sigb^2\), the interaction variance \(\sigma_{\gamma}^2\), and the within-cell variance \(\sigma_e^2\). However, the row variance \(\siga^2\) will not be consistently estimated.
			 \end{enumerate}
			 These considerations motivate requiring all $g\to\infty$, $h\to\infty$ and $m\to\infty$.
			
			The asymptotic representation shows that the influence function of  the maximum likelihood {and REML estimators} under the model is given by the summands of $\bB^{-1}\bphi$. For the case $0 \le \eta < \infty$, this function is $
			\boldsymbol{\mathcal{IF}}=[\mathcal{IF}_{\xi_0},\boldsymbol{\mathcal{IF}}_{\bxi_1}^T,\mathcal{IF}_{\siga^2},\boldsymbol{\mathcal{IF}}_{\bxi_2}^T,\mathcal{IF}_{\sigb^2},\boldsymbol{\mathcal{IF}}_{\bxi_3}^T,\mathcal{IF}_{\sigma_{\gamma}^2},\boldsymbol{\mathcal{IF}}_{\bxi_4}^T,\mathcal{IF}_{\sigma_{e}^2}]^T$,
			where
			\begin{align*}
				\mathcal{IF}_{\xi_0}&=[1+\dsigb^2/\dot\tau(\eta-\eta^{1/2})\barbxbt\bD_{2}^{-1}\barbxb]\sumig\alpha_i-\barbxat\bD_{1}^{-1}\sumig\cxia\alpha_i\\&
				\qquad+\eta[1+(\eta^{-1/2}\dsiga^2/\dot{\tau}+\eta\dsigb^2/\dot{\tau})\barbxbt\bD_{2}^{-1}\barbxb ]\sumjh\beta_j-\eta^{1/2}\barbxbt\bD_{2}^{-1}\sumjh\bxbj\beta_j,
				\\
				\boldsymbol{\mathcal{IF}}_{\bxi_1}&=\bD_{1}^{-1}\sumig\cxia\alpha_i,
				\qquad
				\mathcal{IF}_{\siga^2}=\sumig(\alpha_i^2-\dsiga^2),
				\\
				\boldsymbol{\mathcal{IF}}_{\bxi_2}&=\bD_{2}^{-1}\sumjh\bxbj\beta_j+\bD_{2}^{-1}\barbxb[\dsigb^2/\dot{\tau}(1-\eta^{1/2})\sumig\alpha_i-(\dsiga^2/\dot{\tau}+\eta^{3/2}\dsigb^2/\dot{\tau})\sumjh\beta_j],
					\\
				\mathcal{IF}_{\sigb^2}&=\sumjh(\beta_j^2-\dsigb^2),\qquad
				\boldsymbol{\mathcal{IF}}_{\bxi_3}=\bD_{3}^{-1}\sumig\sumjh\cxijab\gamma_{ij},\qquad
				\mathcal{IF}_{\sigma_{\gamma}^2}=\sumig\sumjh(\gamma_{ij}^2-\dot{\sigma}_{\gamma}^2),\\
				\boldsymbol{\mathcal{IF}}_{\bxi_4}&=\bD_4^{-1}\sumig\sumjh\sumkm\cxijkw  e_{ijk},\qquad
			\mathcal{IF}_{\sigma_{e}^2}=\sumig\sumjh\sumkm(e_{ijk}^2-\dsige^2).
			\end{align*}
			Explicitly, when $\eta=0$  at a point $[\alpha_i,\beta_j,\gamma_{ij},e_{ijk}, \bxait,\bxbjt,\bxabijt,\bxwijkt]^T$ (which we suppress in the notation), 
			the influence function can be simplified to 
			\begin{align*}
					&\mathcal{IF}_{\xi_0}= (1-\barbxa\bD_{1}^{-1}\cxia)\alpha_i,
			&&\boldsymbol{\mathcal{IF}}_{\bxi_1}=\bD_{1}^{-1}\cxia\alpha_i,
			&&\mathcal{IF}_{\siga^2}=\alpha_i^2-\dsiga^2,\\
			&	\boldsymbol{\mathcal{IF}}_{\bxi_2}=\bD_{2}^{-1}\cxjb\beta_j,
				&&\mathcal{IF}_{\sigb^2}=\beta_j^2-\dsigb^2,
			&&	\boldsymbol{\mathcal{IF}}_{\bxi_3}=\bD_{3}^{-1}\cxijab\gamma_{ij},\\
			&	\mathcal{IF}_{\sigma_{\gamma}^2}=\gamma_{ij}^2-\dot{\sigma}_{\gamma}^2,
			&&\boldsymbol{\mathcal{IF}}_{\bxi_4}=\bD_4^{-1}\cxijkw  e_{ijk},
				&&\mathcal{IF}_{\sigma_{e}^2}=e_{ijk}^2-\dsige^2.
			\end{align*}
			These expression are not easy to obtain directly because the row, column, interaction and within-cell parameters are estimated at different rates. As is well-known, the estimators are not robust because the influence function is unbounded in the covariates, random effects, and errors.  Analogous results can be obtained for the case $\eta=\infty$.
			
			The central limit theorem allows us to construct asymptotic confidence intervals for the parameters in the model. 
			Asymptotic $100(1-b)\%$ confidence intervals for $\dot\xi_{1q},\dot\xi_{2q},\dot\xi_{3q},\dot\xi_{4q}$ are given straightforwardly by
			\begin{align*}
				&[ \xi_{1q}-\mathcal{N}_{b}\hat{\sigma}_\alpha\hat{d}^{*1/2}_{1qq}/g^{1/2},  \xi_{1q}+\mathcal{N}_{b}\hat{\sigma}_\alpha\hat{d}^{*1/2}_{1qq}/g^{1/2}], 
				\\
		&[ \xi_{2q}-\mathcal{N}_{b}\hat{\sigma}_\beta\hat{d}^{*1/2}_{2qq}/h^{1/2},  \xi_{2k}+\mathcal{N}_{b}\hat{\sigma}_\beta\hat{d}^{*1/2}_{2qq}/h^{1/2}],  
		\\
		&[ \xi_{3q}-\mathcal{N}_{b}\hat{\sigma}_\gamma\hat{d}^{*1/2}_{3qq}/(gh)^{1/2},  \xi_{3q}+\mathcal{N}_{b}\hat{\sigma}_\gamma\hat{d}^{*1/2}_{3qq}/(gh)^{1/2}], 
		\\
	&[ \xi_{4q}-\mathcal{N}_{b}\hat{\sigma}_e\hat{d}^{*1/2}_{4qq}/n^{1/2},  \xi_{4q}+\mathcal{N}_{b}\hat{\sigma}_e\hat{d}^{*1/2}_{4qq}/n^{1/2}].
			\end{align*}
	For each parameter, $q$ ranges from $1$ to its corresponding $p_a, p_b, p_{ab}$ and $p_w$. Here, $\mathcal{N}_{b}$ denotes the inverse of the cumulative normal distribution function evaluated at $1-b/2$, and the term $\hat{d}^{*}_{iqq}$ represents the $q$th diagonal element of $\hat{\bD}_{i}^{-1}$ for $i=1,2,3,4$, as defined in Corollary \ref{corollary}.
Setting the confidence interval on the log scale and back-transforming, asymptotic $100(1-a)\%$ confidence intervals for $\dsiga, \dsigb, \dsiggama,$ and $\dsige$ are given by
			\begin{align*}
				&[\hat{\sigma}_\alpha\exp\{-\mathcal{N}_{b}(\hat\mu_{4\alpha}-\hat{\sigma}_\alpha^4)^{1/2}/(2g^{1/2}\hat{\sigma}_\alpha^2)\},
				\hat{\sigma}_\alpha\exp\{\mathcal{N}_{b}(\hat\mu_{4\alpha}-\hat{\sigma}_\alpha^4)^{1/2}/(2g^{1/2}\hat{\sigma}_\alpha^2)\}
				],
				\\
				&[\hat{\sigma}_\beta\exp\{-\mathcal{N}_{b}(\hat\mu_{4\beta}-\hat{\sigma}_\beta^4)^{1/2}/(2h^{1/2}\hat{\sigma}_\beta^2) \},
				\hat{\sigma}_\beta\exp\{\mathcal{N}_{b}(\hat\mu_{4\beta}-\hat{\sigma}_\beta^4)^{1/2}/(2h^{1/2}\hat{\sigma}_\beta^2) \}
				],
				\\
				&[\hat{\sigma}_\gamma\exp\{-\mathcal{N}_{b}(\hat\mu_{4\gamma}-\hat{\sigma}_\gamma^4)^{1/2}/(2g^{1/2}h^{1/2}\hat{\sigma}_\gamma^2) \},
				\hat{\sigma}_\gamma\exp\{\mathcal{N}_{b}(\hat\mu_{4\gamma}-\hat{\sigma}_\gamma^4)^{1/2}/(2g^{1/2}h^{1/2}\hat{\sigma}_\gamma^2)\}
				],
				\\
				&[\hat{\sigma}_e\exp\{-\mathcal{N}_{b}(\hat\mu_{4e}-\hat{\sigma}_e^4)^{1/2}/(2n^{1/2}\hat{\sigma}_e^2)\},
				\hat{\sigma}_e\exp\{\mathcal{N}_{b}(\hat\mu_{4e}-\hat{\sigma}_e^4)^{1/2}/(2n^{1/2}\hat{\sigma}_e^2)\}
				],
			\end{align*}
			where $\hat{\sigma}_\alpha^4, \hat{\sigma}_\beta^4, \hat{\sigma}_\gamma^4,$ and $ \hat{\sigma}_e^4$ are in Corollary \ref{corollary}. Squaring the end points provides asymptotic $100(1-b)\%$ confidence intervals for $\dsiga^2, \dsigb^2, \dsiggama^2,$ and $\dsige^2$, respectively.

			 	Our theoretical results show that the row, column, interaction, and within-cell parameters are estimated at different rates. From the structure of $\bF$, even without assuming normality, the maximum likelihood {and REML} estimators for the row and column parameters are independent of both the interaction parameter estimators and the within-cell parameter estimators. Additionally, the interaction parameter estimators are independent of the within-cell parameter estimators. However, there is dependence between the estimators of the row and column parameters when $0 < \eta < \infty$.
			 	This implies that when we view the estimators of the row and column parameters together, the three sets of estimators (row and column, interaction, and within-cell) are asymptotically orthogonal.		
     
			 	For scenarios where $\eta = 0$ or $\eta = \infty$, we can treat the estimators for the row and column parameters separately, with all four sets being asymptotically orthogonal.
			 	Further insights show asymptotic orthogonality between the row slope and its variance, the column slope and its variance, the interaction slope and its variance, and the within-cell slope and its variance. A notable exception exists for the intercept, which is only asymptotically orthogonal to the row (or column) variance when the random effect distribution of the row (or column) is symmetrically centered around zero.
			 	
			 	In working with the model (\ref{two-way crossed model-rw }), we estimate all parameters. If we have no row covariates, we discard $\bxi_1$; if we have no column covariates, we discard $\bxi_2$; if we have no interaction covariates, we discard $\bxi_3$; if we have no within-cell covariates, we discard $\bxi_4$.  The results for these cases can be obtained as special cases of the general results by deleting the components of vectors and the rows and columns of matrices corresponding to the discarded parameter. If there are no  row covariates in the model (there is no $\bxi_1$ in the model), we drop rows and columns 2 to $p_a+1$ from $\bF$ etc. 
     There is a corresponding simplification to Condition A4. Finally, we have treated the covariates in the model as fixed, conditioning on them when they are random. As noted by \cite{yoon2020effect}, when the covariatse are random, it makes sense to treat them as having a similar covariance structure to the response.

		
			\section{Numerical Results}\label{sec:5}
   \begin{table}[!h]
\caption{Simulated coverage (cvge) and average length (len) of confidence intervals based on 1000 replicate analyses when $\alpha_i, \beta_j, \gamma_{ij}$ and $e_{ijk}$ have normal distributions with variance $\siga^2=9$, $\sigb^2=49$, $\siggama^2=36$ and $\sige^2=81$, respectively.}\label{Tab1}
\begin{tabular}{ccccccccccc}
\toprule
\multirow{3}{*}{Estimates}                    &      & \multicolumn{4}{c}{m=10}                            &  & \multicolumn{4}{c}{m=30}                            \\
                                              &      & \multicolumn{2}{c}{h=10} & \multicolumn{2}{c}{h=50} &  & \multicolumn{2}{c}{h=10} & \multicolumn{2}{c}{h=50} \\ \cline{3-6}\cline{8-11}
                                              &      & Cvge        & Len        & Cvge        & Len        &  & Cvge        & Len        & Cvge        & Len        \\\hline
\multirow{2}{*}{$\xi_1$}           & g=10 & 0.80        & 3.34       & 0.90        & 4.85       &  & 0.78        & 4.75       & 0.91        & 3.03       \\
                                              & g=50 & 0.88        & 1.46       & 0.93        & 1.68       &  & 0.89        & 1.48       & 0.92        & 1.62       \\
\multirow{2}{*}{$\sigma_\alpha^2$} & g=10 & 0.82        & 3.14       & 0.88        & 2.80       &  & 0.79        & 2.98       & 0.75        & 1.93       \\
                                              & g=50 & 0.94        & 1.87       & 0.92        & 1.24       &  & 0.97        & 2.17       & 0.94        & 1.38       \\\hline
                                              &      &             &            &             &            &  &             &            &             &            \\\hline
\multirow{2}{*}{$\xi_2$}           & g=10 & 0.90        & 5.96       & 0.95        & 2.22       &  & 0.92        & 8.19       & 0.94        & 2.35       \\
                                              & g=50 & 0.92        & 5.35       & 0.94        & 2.51       &  & 0.91        & 5.40       & 0.95        & 2.43       \\
\multirow{2}{*}{$\sigma_\beta^2$}  & g=10 & 0.71        & 4.00       & 0.94        & 2.83       &  & 0.72        & 3.72       & 0.94        & 2.81       \\
                                              & g=50 & 0.65        & 3.32       & 0.94        & 2.61       &  & 0.68        & 3.73       & 0.92        & 2.54       \\\hline
                                              &      &             &            &             &            &  &             &            &             &            \\\hline
\multirow{2}{*}{$\xi_3$}           & g=10 & 0.92        & 1.46       & 0.92        & 0.49       &  & 0.93        & 1.25       & 0.94        & 0.54       \\
                                              & g=50 & 0.93        & 0.51       & 0.92        & 0.22       &  & 0.94        & 0.50       & 0.95        & 0.23       \\
\multirow{2}{*}{$\sigma_\gamma^2$} & g=10 & 0.82        & 1.58       & 0.90        & 0.83       &  & 0.78        & 1.26       & 0.89        & 0.68       \\
                                              & g=50 & 0.93        & 0.83       & 0.96        & 0.42       &  & 0.88        & 0.68       & 0.95        & 0.35       \\\hline
                                              &      &             &            &             &            &  &             &            &             &            \\\hline
\multirow{2}{*}{$\xi_4$}           & g=10 & 0.96        & 0.38       & 0.95        & 0.18       &  & 0.95        & 0.22       & 0.94        & 0.54       \\
                                              & g=50 & 0.96        & 0.18       & 0.96        & 0.08       &  & 0.94        & 0.10       & 0.96        & 0.23       
                                              \\ \bottomrule
\end{tabular}
\end{table}
 \begin{table}[]
\caption{Simulated coverage (cvge) and average length (len) of confidence intervals based on 1000 replicate analyses when $\alpha_i$ and $\gamma_{ij}$ have normal distributions while $\beta_j$ and $e_{ijk}$ have mixture distributions with variance $\siga^2=9$, $\sigb^2=49$, $\siggama^2=36$ and $\sige^2=81$, respectively.}\label{Tab2}
\begin{tabular}{ccccccccccc}
\toprule
\multirow{3}{*}{Estimates}                    &      & \multicolumn{4}{c}{m=10}                            &  & \multicolumn{4}{c}{m=30}                            \\
                                              &      & \multicolumn{2}{c}{h=10} & \multicolumn{2}{c}{h=50} &  & \multicolumn{2}{c}{h=10} & \multicolumn{2}{c}{h=30} \\ \cline{3-6}\cline{8-11}
                                              &      & Cvge        & Len        & Cvge        & Len        &  & Cvge        & Len        & Cvge        & Len        \\\hline
\multirow{2}{*}{$\xi_1$}           & g=10 & 0.79        & 3.30       & 0.90        & 2.76       &  & 0.82        & 3.46       & 0.90        & 3.65       \\
                                              & g=50 & 0.89        & 1.74       & 0.94        & 1.66       &  & 0.90        & 1.34       & 0.94        & 1.66       \\
\multirow{2}{*}{$\sigma_\alpha^2$} & g=10 & 0.82        & 2.95       & 0.69        & 1.64       &  & 0.80        & 2.85       & 0.74        & 1.83       \\
                                              & g=50 & 0.96        & 1.79       & 0.93        & 1.21       &  & 0.94        & 1.75       & 0.94        & 1.21       \\\hline
                                              &      &             &            &             &            &  &             &            &             &            \\\hline
\multirow{2}{*}{$\xi_2$}           & g=10 & 0.80        & 3.48       & 0.90        & 1.36       &  & 0.80        & 2.56       & 0.92        & 1.88       \\
                                              & g=50 & 0.89        & 3.23       & 0.94        & 1.88       &  & 0.88        & 3.58       & 0.95        & 2.15       \\
\multirow{2}{*}{$\sigma_\beta^2$}  & g=10 & 0.72        & 4.91       & 0.92        & 3.46       &  & 0.71        & 4.74       & 0.91        & 3.51       \\
                                              & g=50 & 0.63        & 4.32       & 0.89        & 3.33       &  & 0.73        & 5.73       & 0.91        & 3.30       \\\hline
                                              &      &             &            &             &            &  &             &            &             &            \\\hline
\multirow{2}{*}{$\xi_3$}           & g=10 & 0.94        & 1.46       & 0.94        & 0.50       &  & 0.95        & 1.29       & 0.95        & 0.56       \\
                                              & g=50 & 0.95        & 0.49       & 0.94        & 0.22       &  & 0.95        & 0.54       & 0.95        & 0.23       \\
\multirow{2}{*}{$\sigma_\gamma^2$} & g=10 & 0.81        & 1.31       & 0.91        & 0.70       &  & 0.78        & 1.17       & 0.88        & 0.64       \\
                                              & g=50 & 0.91        & 0.70       & 0.95        & 0.36       &  & 0.88        & 0.64       & 0.94        & 0.33       \\\hline
                                              &      &             &            &             &            &  &             &            &             &            \\\hline
\multirow{2}{*}{$\xi_4$}           & g=10 & 0.96        & 0.25       & 0.95        & 0.12       &  & 0.96        & 0.15       & 0.96        & 0.56       \\
                                              & g=50 & 0.95        & 0.12       & 0.96        & 0.05       &  & 0.95        & 0.06       & 0.95        & 0.23    
                                              \\ \bottomrule
\end{tabular}
\end{table}
We carried out simulations to explore the accuracy of the asymptotic results.  We generated data for $g \in \{10, 50\}$, $h \in \{10,50\}$, and $m \in \{10, 30\}$.  The covariate $x_{ijk}$ was generated with a crossed structure by setting $x_{ijk} = 4 +t_i+ 1.5u_j + 2v_{ij} + 3w_{ijk}$, where $t_i$, $u_j$, $v_{ij}$, and $w_{ijk}$ are independent standard normal random variables.  The random effects were generated independently with $\alpha_i$ generated from $F_{\alpha}$ with $\E(\alpha_i) = 0$ and $\var(\alpha_i) = \siga^2$, $\beta_j$ from  $F_{\beta}$ with $\E(\beta_j) = 0$ and $\var(\beta_j) = \sigb^2$, $\gamma_{ij}$ from $F_{\gamma}$ with $\E(\gamma_{ij}) = 0$ and $\var(\gamma_{ij}) = \siggama^2$, and $e_{ijk}$ from $F_{e}$ with $\E(e_{ijk}) = 0$ and $\var(e_{ijk}) = \siggama^2$.   We set $\dsiga^2 =9$ with $F_\alpha=N(0,\dsiga^2)$,
 $\dsigb^2 =49$ with  $F_\beta=N(0,\dsigb^2)$ or $F_\beta=0.3N(0.5,1)+0.7N(-0.3\times0.5/0.7,(\dsigb^2-0.375-0.7\mu^2)/0.7)$,
 $\dsiggama^2=36$ with  $F_\gamma=N(0,\dsiggama^2)$, and
 $\dsige^2=81$ with  $F_e=N(0,\dsige^2)$ or $F_e=0.3N(0.5,1)+0.7N(-0.3\times0.5/0.7,(\dsige^2-0.375-0.7\mu^2)/0.7)$.  Here $\mu=-0.3\times 0.5/0.7$.
 We then computed the response using the model
\begin{align*}
 y_{ijk}=&\xddd\xi_0+(\xidd-\xddd)\xi_1+(\xdjd-\xddd)\xi_2+(\xijd-\xidd-\xdjd+\xddd)\xi_3+(x_{ijk}-\xijd)\xi_4    \\&
 +\alpha_i+\beta_j+\gamma_{ij}+e_{ijk},\qquad i=1,\ldots,g, j=1,\ldots,h, k=1,\ldots,m,
\end{align*}
with the true parameters to $\dot\bxi = [0,5,7,3,4]^T$. We explored 32 scenarios resulting from the combination of the 8 $(g,h,m)$ configurations with the 4 additional random effect settings. 

For each simulation setting, we generated 1000 datasets and applied the model using REML through the \texttt{lmer} function in the R package \textit{lme4}. We computed asymptotic standard errors (SE), constructed $95\%$ confidence intervals as described in Section~\ref{Sec:4}, and computed the coverage (with Monte Carlo standard error $<0.006$) and mean  length of each of these intervals.

Tables~\ref{Tab1} and \ref{Tab2} show the estimated confidence interval coverage (Cvge) and mean length (Len) over the replicated datasets for the settings with $\siga^2=9$, $\sigb^2=49$, $\siggama^2=36$ and $\sige^2=81$ and both \(\beta_j\) and \(e_{ijk}\) under normal (Table~\ref{Tab1}) and mixture distributions (Table~\ref{Tab2}).  As we expect, the standard errors for the estimates of the within-cell slope $\xi_4$ are identical and so the confidence intervals attain their nominal coverage, even when $g=10, h=10$ and $m=10$. This shows the advantage of the faster rate of convergence depending on $n$ (which is already $>1000$) rather than $g$ or $h$ or $gh$. 
For the row slope $\xi_1$, the confidence interval is below the nominal coverage when \(g = 10\), regardless of \(h = 10\) or \(50\), and \(m = 10\) or \(30\). However, as $g$ increases to $50$, the coverage also increases, approaching or even reaching the nominal coverage. The standard deviation $\siga$ follows the same pattern: with $g=10$ the initial undercoverage is greater, but as $g$ increases, the convergence to the nominal coverage becomes more dramatic. This also supports our theoretical asymptotic result---that the rate of convergence for the row parameters depends on $g$ rather than $h$ or $m$.
For the column slope $\xi_2$, the confidence interval is consistently below the nominal coverage for \(h = 10\), regardless of \(g = 10\) or \(50\), or \(m = 10\) or \(30\).  But as $h$ increases to $50$, the coverage also increases, approaching or reaching the nominal coverage. The standard deviation $\sigb$ follows the same pattern: with \(h = 10\), the undercoverage is more pronounced, but approaches the nominal coverage as \(h\) increases. Again, this result also supports our theoretical asymptotic result that the rate of convergence for the between-column parameters depends on $h$ rather than $g$ or $m$.
For the interaction slope $\xi_3$, the confidence interval is close to or has reached the nominal coverage with \(g = 10\) and \(h = 10\), regardless of \(m = 10\) or \(30\).  
For the interaction variance $\sigma_{\gamma}$, the coverage is substantially below the nominal level with \(g = 10\) and \(h = 10\), regardless of \(m = 10\) or \(30\). However, as either \(g\) or \(h\) increases, the coverage rises, nearing the nominal rate. When both \(g\) and \(h\) are at \(30\), the coverage attains the nominal level.
These results support our theoretical asymptotic result -- that the rate of convergence for the between-cell parameters depends on the product $gh$, which is already greater than $100$.

 
\section{Proofs}\label{Proof}
The proofs of Theorems \ref{theorem1} and \ref{theorem3} are presented in Sections~\ref{proof1} and \ref{proof-theorem3}, respectively. The supporting lemmas used in these proofs are  proved in Sections~\ref{proof2} and ~\ref{proof3}.
   
\subsection{Proof of Theorem \ref{theorem1}}\label{proof1}
			
\begin{proof}
Write 
\begin{equation}\label{mainproof eq1}
\begin{split}
								\bK^{-1/2}\bpsi(\bomega)=&\bK^{-1/2}\bphi-\bB\bK^{1/2}(\bomega-\dot\bomega)   + T_1+T_2(\bomega)+T_3(\bomega),
\end{split}
\end{equation}
where  $\bB=\lim_{g,h,m\to\infty}-\bK^{-1/2}\E\triangledown\bpsi (\dot{\bomega})\bK^{-1/2}$, $T_1=\bK^{-1/2}[\bpsi(\dot\bomega)-\bphi(\dot\bomega)]$,
$T_2(\bomega)=\bK^{-1/2}\E[\bpsi(\bomega)-\bpsi(\dot\bomega)] + \bB \bK^{1/2}(\bomega-\dot\bomega)$ and $T_3(\bomega)=\bK^{-1/2}[\bpsi(\bomega)-\bpsi(\dot\bomega)-\E\{\bpsi(\bomega)-\bpsi(\dot\bomega)\}]$. 
 For $0<M<\infty$, let  $\boldsymbol{N}=\{\bomega: |\bK^{1/2}(\bomega-\dot\bomega)|\le M\}$ define a local (shrinking) ball centered at $\dot{\bomega}$.  If we can show  that $|T_1|=o(1)$, $\underset{\bomega\in \boldsymbol{N}}{\sup}|T_2(\bomega)|=o_p(1)$ and $\underset{\bomega\in \boldsymbol{N}}{\sup}|T_3(\bomega)|=o_p(1)$,
 respectively, then uniformly on $|\bK^{1/2}(\bomega-\dot\bomega)|\le M$, we have
				\begin{equation}\label{lastone}
					\bK^{-1/2}\bpsi(\bomega)=\bK^{-1/2}\bphi- \bB \bK^{1/2}(\bomega-\dot\bomega)+o_p(1).
				\end{equation}
Multiplying by $(\bomega-\dot\bomega)^T\bK^{1/2}$, 
				\begin{equation*}
					(\bomega-\dot\bomega)^T\bpsi(\bomega)=(\bomega-\dot\bomega)^T\bphi- (\bomega-\dot\bomega)^T\bK^{1/2}\bB \bK^{1/2}(\bomega-\dot\bomega)+o_p(1).
				\end{equation*}
				Since $\bB$ is positive definite, the right-hand side of (\ref{lastone}) is negative for $M$ sufficiently large. Therefore, according to Result 6.3.4 of \cite{ortega1970iterative}, a solution to the estimating equations exists in probability and satisfies $|\bK^{-1/2}(\hbomega-\dot\bomega)|=O_p(1)$, so $\hbomega \in \boldsymbol{N}$.  This allows us to substitute  $\hbomega$ for $\bomega$ in (\ref{lastone}) and rearrange the terms to obtain the asymptotic representation for $\hbomega$; the central limit theorem then follows from Lemma \ref{lema3}.
				
				It remains to show that the remainder  terms $T_1$, $T_2(\bomega)$, and $T_3(\bomega)$ in (\ref{mainproof eq1}) are of sufficiently small order that they can be ignored. In Lemma~\ref{lema1},  we establish $|T_1|=o_p(1)$ by showing that the result holds for each component of $T_1=\bK^{-1/2}[\bpsi(\dot\bomega)-\bphi(\dot\bomega)]$, by applying Chebychev's inequality, and calculating the variances of the components.
				
				The method used for addressing  $T_2(\bomega)$ and $T_3(\bomega)$ is inspired by \cite{bickel1975one} who applied similar arguments to one-step regression estimators. The approach was extended to maximum likelihood and REML estimators in linear mixed models by \cite{richardson1994asymptotic, lyu2019estimation}.  For $T_2(\bomega)$, we have 
				\begin{align*}
					\underset{\bomega\in \boldsymbol{N}}{\sup}  |T_2(\bomega)|&\le
					\underset{\bomega\in \boldsymbol{N}}{\sup} |\bK^{-1/2}[\E\{\bpsi(\bomega)-\bpsi(\dot\bomega)\}-\E\triangledown\bpsi(\dot\bomega)(\bomega-\dot\bomega)]|\\&
					\quad+
					\underset{\bomega\in \boldsymbol{N}}{\sup} |\bK^{-1/2}\E\triangledown\bpsi(\dot\bomega)(\bomega-\dot\bomega)+\bB\bK^{1/2}(\bomega-\dot\bomega)|\\&
					\le 	
						\underset{\bomega\in \boldsymbol{N}}{\sup} |\bK^{-1/2}[\E\triangledown\bpsi(\bOmega)-\E\triangledown\bpsi(\dot\bomega)](\bomega-\dot\bomega)|
					+
					M  \|-\bK^{-1/2}\E\triangledown\bpsi(\dot\bomega)\bK^{-1/2}-\bB\|\\&						
				\le 	
				M\underset{\bomega\in \boldsymbol{N}}{\sup} \|\bK^{-1/2}[\E\triangledown\bpsi(\bOmega)-\E\triangledown\bpsi(\dot\bomega)]\bK^{-1/2}\|
				+
				M  \|\bB_n-\bB\|,
								\end{align*}
				where the $i$th row of the derivative matrix $\triangledown\bpsi$ is evaluated at the $i$th row of the matrix $\bOmega$ and each row of $\bOmega$ lies between $\bomega$ and $\dot\bomega$, and $\bB_n=-\bK^{-1/2}\E\triangledown\bpsi(\dot\bomega)\bK^{-1/2}$.  In Lemma~\ref{lma Bn},  it is demonstrated that  $\lim_{g,h,m\to\infty}\|\bB_n-\bB\|=o(1)$.  Additionally,  in  Lemma~\ref{lma5}, we show
				$
					M{\sup}_{\bomega\in \boldsymbol{N}}\|\bK^{-1/2} [\E\triangledown\bpsi(\bOmega)-\E\triangledown\bpsi(\dot\bomega)]\bK^{-1/2}\|=o(1)$.
		 Moving to $T_3(\bomega)$,  we decompose $\boldsymbol{N}=\{\bomega: |\bK^{1/2}(\bomega-\dot\bomega)|\le M\}$ into the set of $N=O(g^{1/4})$ smaller cubes
					$\mathcal{C}=\{\mathcal{C}(\boldsymbol{\zeta}_k)\}$, where $\mathcal{C}(\boldsymbol{\zeta})= \{\bomega :|\bK^{1/2}(\boldsymbol{\zeta}-\dot\bomega)|\le Mg^{-\frac{1}{4}}\}$.  We first show that
					$|T_3(\bomega)|$ holds over the  set of indices. Using Chebyshev's inequality, for any $\eta>0$, we have 
					
					\begin{align*}
						&\sum_{k=1}^N\pr \left( |\bK^{-1/2}[\bpsi(\boldsymbol{\zeta}_k)-\bpsi(\dot\bomega)-\E\{\bpsi(\boldsymbol{\zeta}_k)-\bpsi(\dot\bomega)\}]|>\eta\right) \\&
						\le\eta^{-2}\sum_{k=1}^N\E|\bK^{-1/2}[\bpsi(\boldsymbol{\zeta}_k)-\bpsi(\dot\bomega)-\E\{\bpsi(\boldsymbol{\zeta}_k)-\bpsi(\dot\bomega)\}]|^2\\&	
						=
						\eta^{-2}g^{-1}\sum_{k=1}^N\E|\bpsi^{(a)}(\boldsymbol{\zeta}_k)-\bpsi^{(a)}(\dot\bomega)-\E\{\bpsi^{(a)}(\boldsymbol{\zeta}_k)-\bpsi^{(a)}(\dot\bomega)\}|^2\\&	\quad+
						\eta^{-2}h^{-1}\sum_{k=1}^N\E|\bpsi^{(b)}(\boldsymbol{\zeta}_k)-\bpsi^{(b)}(\dot\bomega)-\E\{\bpsi^{(b)}(\boldsymbol{\zeta}_k)-\bpsi^{(b)}(\dot\bomega)\}|^2\\&	\quad+
						\eta^{-2} (gh)^{-1}\sum_{k=1}^N\E|\bpsi^{(ab)}(\boldsymbol{\zeta}_k)-\bpsi^{(ab)}(\dot\bomega)-\E\{\bpsi^{(ab)}(\boldsymbol{\zeta}_k)-\bpsi^{(ab)}(\dot\bomega)\}|^2\\&\quad+	
						\eta^{-2}n^{-1}\sum_{k=1}^N\E|\bpsi^{(w)}(\boldsymbol{\zeta}_k)-\bpsi^{(w)}(\dot\bomega)-\E\{\bpsi^{(w)}(\boldsymbol{\zeta}_k)-\bpsi^{(w)}(\dot\bomega)\}|^2\\&	
						=	\eta^{-2}g^{-1}\sum_{k=1}^N\operatorname{trace}[\var\{\bpsi^{(a)}(\boldsymbol{\zeta}_k)-\bpsi^{(a)}(\dot\bomega)\}]
						+\eta^{-2}h^{-1}\sum_{k=1}^N\operatorname{trace}[\var\{\bpsi^{(b)}(\boldsymbol{\zeta}_k)-\bpsi^{(b)}(\dot\bomega)\}]\\&\quad
						+\eta^{-2} (gh)^{-1}\sum_{k=1}^N\operatorname{trace}[\var\{\bpsi^{(ab)}(\boldsymbol{\zeta}_k)-\bpsi^{(w)}(\dot\bomega)\}]
						+\eta^{-2}n^{-1}\sum_{k=1}^N\operatorname{trace}[\var\{\bpsi^{(w)}(\boldsymbol{\zeta}_k)-\bpsi^{(w)}(\dot\bomega)\}].
					\end{align*}
					We show in  Lemma \ref{lma3} that all variance terms are uniformly bounded by a finite constant $L$,  so 
					\begin{align*}
				&\sum_{k=1}^N\pr \left( |\bK^{-1/2}[\bpsi(\boldsymbol{\zeta}_k)-\bpsi(\dot\bomega)-\E\{\bpsi(\boldsymbol{\zeta}_k)-\bpsi(\dot\bomega)\}]|>\eta\right) \\&
				\le \eta^{-2}g^{-1}N(p_a+2)L+\eta^{-2}h^{-1}N(p_b+1)L+\eta^{-2} (gh)^{-1}N(p_{ab}+1)L+\eta^{-2}n^{-1}N(p_w+1)L\\&
					=o(1),
					\end{align*}
			  using fact that $N=O(g^{1/4})$.	Using Taylor expansion, we get
					\begin{align*}
						&\underset{1\le k\le N}{\max}\underset{\bomega\in\mathcal{C}(\boldsymbol{\zeta}_k)}{\sup}|\bK^{-1/2}[\bpsi(\bomega)-\bpsi(\boldsymbol{\zeta}_k)-\E\{\bpsi(\bomega)-\bpsi(\boldsymbol{\zeta}_k)\}]|\\&
						= \underset{1\le k\le N}{\max}\underset{\bomega\in\mathcal{C}(\boldsymbol{\zeta}_k)}{\sup}|\bK^{-1/2}\{\nabla\bpsi(\bOmega_k)(\bomega - \boldsymbol{\zeta}_k)-\E\nabla\bpsi(\bOmega_k)(\bomega-\boldsymbol{\zeta}_k)\}|
					\\&
					\le M \underset{\bomega\in\boldsymbol{N}}{\sup} g^{-1/4}\|\bK^{-1/2}\{\nabla\bpsi(\bOmega)-\E\nabla\bpsi(\bOmega)\}\bK^{-1/2}\|,
					\end{align*}
				where the rows of $\bOmega_k$ are between $\boldsymbol{\zeta}_k$ and $\bomega$. 
The result follows from Lemma \ref{lma6}.  
			\end{proof}
  			\subsection{Proof of Theorem \ref{theorem3} }\label{proof-theorem3}
We write 
\begin{align*}
				&|g^{-1/2}\{l_{A\siga^2}(\bomega)-l_{\siga^2}(\bomega)\}|\le\frac{1}{2g^{1/2}}|\operatorname{trace}\{ (\bX^T\bV^{-1}\bX)^{-1}\bX^T\bV^{-1}\bZ_1\bZ_1^T\bV^{-1}\bX\}|,\\
				&|h^{-1/2}\{l_{A\sigb^2}(\bomega)-l_{\sigb^2}(\bomega)\}|\le\frac{1}{2h^{1/2}}|\operatorname{trace}\{ (\bX^T\bV^{-1}\bX)^{-1}\bX^T\bV^{-1}\bZ_2\bZ_2^T\bV^{-1}\bX\}|,\\
				&|g^{-1/2}h^{-1/2}\{l_{A\siggama^2}(\bomega)-l_{\siggama^2}(\bomega)\}|\le\frac{1}{2g^{1/2}h^{1/2}}|\operatorname{trace}\{ (\bX^T\bV^{-1}\bX)^{-1}\bX^T\bV^{-1}\bZ_3\bZ_3^T\bV^{-1}\bX\}|,\\
				&|n^{-1/2}\{l_{A\sige^2}(\bomega)-l_{\sige^2}(\bomega)\}|\le \frac{1}{2n^{1/2}}|\operatorname{trace}\{ (\bX^T\bV^{-1}\bX)^{-1}\bX^T\bV^{-1}\bZ_0\bZ_0^T\bV^{-1}\bX\}|.
			\end{align*}
 From Supplementary Section S.2.7, for $g,h,m$ sufficiently large,  uniformly on $\bomega\in \boldsymbol{N}$, we have 
$|\operatorname{trace}\{ (\bX^T\bV^{-1}\bX)^{-1}\bX^T\bV^{-1}\bZ_i\bZ_i^T\bV^{-1}\bX\}|=O(1)$, {for $i=0,1,2,3$}. Consequently,
$|g^{-1/2}\{l_{A\siga^2}(\bomega)-l_{\siga^2}(\bomega)\}|=o_p(1)$,
$|h^{-1/2}\{l_{A\sigb^2}(\bomega)-l_{\sigb^2}(\bomega)\}|=o_p(1)$,
$|g^{-1/2}h^{-1/2}\{l_{A\siggama^2}(\bomega)-l_{\siggama^2}(\bomega)\}|=o_p(1)$,
and $|n^{-1/2}\{l_{A\sige^2}(\bomega)-l_{\sige^2}(\bomega)\}|=o_p(1)$ uniformly on $\bomega\in \boldsymbol{N}$, and the result follows from Theorem~\ref{theorem1}.

\subsection{Lemmas for the estimating function}\label{proof2}
		We present lemmas with outline proofs to show that \(T_1 = o_p(1)\), \(\bpsi(\dot{\bomega})\) approximates \(\bphi\),  a central limit theorem holds for \(\bphi\), and  the variances of \(\bpsi(\bomega) - \bpsi(\dot\bomega)\) are uniformly bounded. Detailed proofs can be found in the Supplementary Section S.6.
\begin{lma}\label{lema1}
	Suppose Condition A holds. Then $|\bK^{-\frac{1}{2}}\{\bpsi(\dot\bomega)-\bphi\}|=o_p(1)$. 
\end{lma}

\begin{proof}
	We establish that each component of $\bK^{-\frac{1}{2}}\{\bpsi(\dot\bomega)-\bphi\}$ converges to zero in probability. For example,	we have 
	\begin{align*}
		l_{\xi_0}(\dot\bomega)-\phi_{\xi_0}=({gm}/{\dot\lambda_4}-{1}/{\sigmacom})(\sumig\alpha_i+g\bar\beta)+({ghm}/{\dot\lambda_4})(\gammadd-\eddd),
	\end{align*}
 where $\bar\beta$, $\gammadd$, and $\eddd$ are defined in Appendix A.
	Then we have $g^{-1/2} \{l_{\xi_0}(\dot\bomega)-\phi_{\xi_0}\}$ has mean zero and variance
	\begin{align*}
		&\var[g^{-1/2}\{l_{\xi_0}(\dot\bomega)-\phi_{\xi_0}\}]\\&=g^{-1}({hm}/{\dlambda_4}-1/\sigmacom)^2\{\sumig\var(\alpha_i)+g^2\var(\bar\beta)\}+({gh^2m^2}/{\dlambda_4^2})\{\var(\gammadd)+\var(\eddd)\}\\&
		=O((h\wedge g)^{-2}),
	\end{align*}
	because ${hm}/{\dlambda_4}-{1}/{\sigmacom}=O((h\wedge g)^{-1})$, where $a \wedge b$ represents minimum value of $a$ and $b$.
	Thus, as $g,h,m\rightarrow\infty$, we obtain 
	\begin{align*}
		g^{-1/2} \{l_{\xi_0}(\dot\bomega)-\phi_{\xi_0}\}=o_p(1).
	\end{align*}
	 Similarly, we can obtain
	\begin{align*}
		&g^{-1/2}\{l_{\bxi_1}(\dot\bomega)-\bphi_{\bxi_1}\}=\bo_p(1),
		&&g^{-1/2}\{l_{\sigma_{\alpha}^2}(\dot\bomega )-\phi_{\sigma_{\alpha}^2}\}=o_p(1),
		&&h^{-1/2}\{l_{\bxi_2}(\dot\bomega)-\bphi_{\bxi_2}\}=\bo_p(1), \\
		&h^{-1/2}\{l_{\sigma_{\beta}^2}(\dot\bomega )-\phi_{\sigma_{\beta}^2}\}=o_p(1),
				&& (gh)^{-1/2}\{l_{\bxi_3}(\dot\bomega)-\bphi_{\bxi_3}\}=\bo_p(1), 
		&& (gh)^{-1/2}\{l_{\sigma_{\gamma}^2}(\dot\bomega )-\phi_{\sigma_{\gamma}^2}\}=o_p(1),\\
		&n^{-1/2}(l_{\bxi_4}(\dot\bomega)-\bphi_{\bxi_4})=\bo_p(1), 
		&&n^{-1/2}\{l_{\sigma_{e}^2}(\dot\bomega )-\phi_{\sigma_{e}^2}\}=o_p(1),
	\end{align*}
	where $\bo_p(1)$ represents a vector with all elements having $o_p(1)$. Full detail can be found in Supplementary Section S.6.1.
\end{proof}
\begin{lma}\label{lema3}
	Suppose Condition A holds. Then, as $g,h,m\to\infty$, $\bK^{-1/2}\bphi\xrightarrow{D}N(\bzero,\bA)$, where $\bA=\diag(\bA_1,\bA_2,\bA_3)$ with
	\begin{align*}
		&\bA_{1}
		=	\resizebox{1.1\linewidth}{!}{%
			$1/\dot{\tau}\begin{bmatrix}
			1&{\barbxat}&{\E\alpha_1^3}/(2\dot{\sigma}_\alpha^4)&{\eta^{1/2}\barbxbt}&{\eta^{1/2}\E \beta_1^3}/{(2\dot{\sigma}_\beta^4)}\\
			{\barbxa}&{\dot{\tau}\bD_1}/\dot{\sigma}_\alpha^2+{\barbxa\barbxat}&{\E\alpha_1^3\barbxa}/{(2\dot{\sigma}_\alpha^4)}&{\eta^{1/2}\barbxa\barbxbt}&{\eta^{1/2}\E\beta_1^3\barbxa}/{(2\dot{\sigma}_\beta^4)}\\
			{\E\alpha_1^3}/{(2\dot{\sigma}_\alpha^4 )}&{\E\alpha_1^3\barbxat}/{(2\dot{\sigma}_\alpha^4\dot{\tau})}&\dot{\tau}({\E\alpha_1^4-\dot{\sigma}_\alpha^4})/({4\dot{\sigma}_\alpha^8})&{\eta^{1/2}\E\alpha_1^3 \barbxbt}/({2\dot{\sigma}_\alpha^4})&0\\
			{\eta^{1/2}\barbxb}&{\eta^{1/2}\barbxb\barbxat}&{\eta^{1/2}\E\alpha_1^3 \barbxb}/({2\dot{\sigma}_\alpha^4})&{\dot{\tau}{\bD}_2}/{\dot{\sigma}_\beta^2}+ {\eta\barbxb\barbxbt}&{\eta E\beta_1^3\barbxb}/({2\dot{\sigma}_\beta^4})\\
			{\eta^{1/2}\E \beta_1^3}/(2\dot{\sigma}_\beta^4)&{\eta^{1/2}\E\beta_1^3\barbxat}/({2\dot{\sigma}_\beta^4})&0&{\eta E\beta_1^3\barbxbt}/{(2\dot{\sigma}_\beta^4)}&\dot{\tau}({\E\beta_1^4-\dot{\sigma}_\beta^4})/({4\dot{\sigma}_\beta^8})			
		\end{bmatrix},$}
		\\
		& 
		\bA_2=\begin{bmatrix}
			{\bD_3}/{\dot{\sigma}_\gamma^2}&\bzero_{[p_{ab}:1]}\\
			\bzero_{[1:p_{ab}]}&({\E\gamma_{11}^4-\dot{\sigma}_\gamma^4}/){4\dot{\sigma}_\gamma^8}
		\end{bmatrix}
		\quad\text{and}\quad
		\bA_3=\begin{bmatrix}
			{\bD_4}/{\dot{\sigma}_e^2}&\bzero_{[p_w:1]}\\
			\bzero_{[1:p_w]}&({\E e_{111}^4-\dot{\sigma}_e^4})/{(4\dot{\sigma}_e^8)}
		\end{bmatrix}.
	\end{align*}
	
\end{lma}
\begin{proof}
	
	
Consider a partition of $\bphi$ into the following components: $\bphi^{(a)}$, which contains $p_a + 2$ elements corresponding to the between-row parameters; $\bphi^{(b)}$, containing $p_b + 1$ elements corresponding to the between-column parameters; $\bphi^{(ab)}$, encompassing $p_{ab} + 1$ elements corresponding to the between-cell parameters; and $\bphi^{(w)}$, including $p_w + 1$ elements corresponding to the within-cell parameters.

	We aim to demonstrate that $(g^{-1/2}\bphi^{(a)T},h^{-1/2}\bphi^{(b)T})^T \xrightarrow{D} N(\bzero, \bA_1)$, $ (gh)^{-1/2}\bphi^{(ab)} \xrightarrow{D} N(\bzero, \bA_2)$, and $n^{-1/2}\bphi^{(w)} \xrightarrow{D} N(\bzero, \bA_3)$. The results subsequently follow from the independence of $(\bphi^{(a)T}, \bphi^{(b)T})^T$, $\bphi^{(ab)}$, and $\bphi^{(w)}$. 
	
	We let   $(\bphi^{(a)T},\bphi^{(b)T})^T =\sumig\sumjh (\bphi_{ij}^{(a)T},\bphi_{ij}^{(b)T})^T$, with the summands \\
 $\bphi_{ij}^{(a)}=h^{-1}(\phi_{ij\xi_0},\bphi_{ij\bxi_1}^T,\phi_{ij\sigma_{\alpha}^2})^T$
	and $\bphi_{ij}^{(b)}=g^{-1}(\bphi_{ij\bxi_2}^T,\phi_{ij\sigma_{\beta}^2})^T$, and let $\ba=(a_1,\ba_2^T,a_3,\ba_4^T,a_4)^T$ be a fixed $(p_a+p_b+3)$-vector satisfying $\ba^T\ba=1$. Then $\ba^T  (g^{-1/2}\bphi^{(a)T},h^{-1/2}\bphi^{(b)T})^T$ is a sum of independent scalar random variables with mean zero and finite variance. As $g,h\to\infty$, $\var[\ba^T (g^{-1/2}\bphi^{(a)T},h^{-1/2}\bphi^{(b)T})^T]\rightarrow\ba^T\bA_{1}\ba$.   
	Moreover, according to Conditions A3 and A4, and  the absolute moment inequality \cite{von1965inequalities}, there exists a positive finite value $\delta$ , such that 
	\begin{align*}
		&\sumig\sumjh \E|\ba^T  (g^{-1/2}\bphi_{ij}^{(a)T}, h^{-1/2}\bphi_{ij}^{(b)T})^T|^{2+\delta}\\&
		\le 5^{1+\delta} g^{-1-\delta/2}h^{-2-\delta}\sumig\sumjh \{
   \E |a_1\phi_{ij\xi_0}|^{2+\delta}+  \E |\ba_2^T\bphi_{ij\bxi_1}|^{2+\delta}+  \E |a_3\phi_{ij\sigma_{\alpha}^2}|^{2+\delta}\}\\&\quad
  +5^{1+\delta} g^{-2-\delta} h^{-1-\delta/2}\sumig\sumjh \{\E |\ba_4^T\bphi_{ij\bxi_2}|^{2+\delta}+ \E |a_5\phi_{ij\sigma_{\beta}^2}|^{2+\delta}\} 
	\end{align*}
	From Equation (\ref{eq:approx}), we have
	\begin{align*}
		&\phi_{ij\xi_0}=\frac{1}{\sigmacom}(\alpha_i+\beta_j),
		\qquad \bphi_{ij\bxi_1}=\frac{1}{\dsiga^2}\cxia\alpha_i+\frac{\barbxa}{\sigmacom}(\alpha_i+\beta_j),
\qquad
		\phi_{ij\sigma_{\alpha}^2}=\frac{1}{2\dot{\sigma}_\alpha^4}(\alpha_i^2-\dot{\sigma}_\alpha^2),\\
		& \bphi_{ij\bxi_2}=\frac{1}{\dsigb^2}\cxjb\beta_j+\frac{\eta\barbxb}{\sigmacom}(\alpha_i+\beta_j),
		\qquad
		 \phi_{ij\sigma_{\beta}^2}=\frac{1}{2\dot{\sigma}_\beta^4}(\beta_j^2-\dot{\sigma}_\beta^2).
	\end{align*}
	We then obtain
	\begin{align*}
		&\E |a_1\phi_{ij\xi_0}|^{2+\delta}\le  2^{1+\delta}\dot{\tau}^{-2-\delta}(\E|\alpha_1|^{2+\delta}+\E|\beta_1|^{2+\delta}),
		\\&\E |\ba_2^T\bphi_{ij\bxi_1}|^{2+\delta}\le 3^{1+\delta}\{\dsiga^{-4-2\delta}|\cxia|^{2+\delta}\E|\alpha_1|^{2+\delta}+\dot{\tau}^{-2-\delta}|\barbxa|^{2+\delta}(\E|\alpha_1|^{2+\delta}+\E|\beta_1|^{2+\delta})\},
		\\&\E |a_3\phi_{ij\sigma_{\alpha}^2}|^{2+\delta}\le 2^{1+\delta}(2\dot{\sigma}_{\alpha})^{-8-4\delta}(\E|\alpha_1|^{4+2\delta}+\dot{\sigma}_{\alpha}^{4+2\delta}),
		\\&\E |\ba_4^T\bphi_{ij\bxi_2}|^{2+\delta}\le 3^{1+\delta}\{
		\dsigb^{-4-2\delta}|\cxjb|^{2+\delta}\E|\beta_1|^{2+\delta}+\dot{\tau}^{-2-\delta}\eta^{2+\delta}|\barbxb|^{2+\delta}(\E|\alpha_1|^{2+\delta}+\E|\beta_1|^{2+\delta})\},
		\\&\E |a_5\phi_{ij\sigma_{\beta}^2}|^{2+\delta}\le 2^{1+\delta}(2\dot{\sigma}_{\beta})^{-8-4\delta}(\E|\beta_1|^{4+2\delta}+\dot{\sigma}_{\beta}^{4+2\delta}).
	\end{align*}
	Thus, by Conditions A3 and A4, we deduce
	\begin{align*}
		(\ba^T\bA_{1}\ba)^{-2-\delta} \sumig\sumjh \E|\ba^T \bK_1^{-1/2}(\bphi_{ij}^{(a)T},\bphi_{ij}^{(b)T})^T|^{2+\delta}=o_p(1),
	\end{align*}
where $\bK_1=\diag(g\bI_{p_a+2},h\bI_{p_b+1})$, 	
 and Lyapunov's condition is satisfied. 
	As a result, $\ba^T \bK_1^{-1/2}(\bphi^{(a)T},\bphi^{(b)T})^T$ converges to $N(0, \boldsymbol{a}^T\boldsymbol{A}_{1}\boldsymbol{a})$ as $g \to \infty$, and  the conclusion follows from the Cramer-Wold device \cite[p49]{billingsley1968convergence}.	
	The  proof that $ (gh)^{-1/2}\bphi^{(ab)} \xrightarrow{D} N(\bzero, \bA_2)$ as  $g, h\to \infty$, and $n^{-1/2}\bphi^{(w)} \xrightarrow{D} N(\bzero, \bA_3)$  as $g,h,m \to \infty$, are similar. Details can be found in Supplementary Section S.6.2.
\end{proof}
\begin{lma}\label{lma3}
	Suppose Condition A holds. Then, there is a finite constant $L$ such that
	\begin{align*}
		&	\underset{\bomega\in \boldsymbol{N}}{\sup} \var\{l_{\xi_0}(\bomega)-l_{\xi_0}(\dot\bomega)\}\le L,
		&&	\underset{\bomega\in \boldsymbol{N}}{\sup} \var\{l_{\xi_{1q}}(\bomega)-l_{\xi_{1q}}(\dot\bomega)\}\le L,
		&&	\underset{\bomega\in \boldsymbol{N}}{\sup} \var\{l_{\sigma_{\alpha}^2}(\bomega)-l_{\sigma_{\alpha}^2}(\dot\bomega)\}\le L,\\
		&	\underset{\bomega\in \boldsymbol{N}}{\sup} \var\{l_{\xi_{2r}}(\bomega)-l_{\xi_{2r}}(\dot\bomega)\}\le L,
		&&	\underset{\bomega\in \boldsymbol{N}}{\sup} \var\{l_{\sigma_{\beta}^2}(\bomega)-l_{\sigma_{\beta}^2}(\dot\bomega)\}\le L,
		&&	\underset{\bomega\in \boldsymbol{N}}{\sup} \var\{l_{\xi_{3s}}(\bomega)-l_{\xi_{3s}}(\dot\bomega)\}\le L,\\
		&	\underset{\bomega\in \boldsymbol{N}}{\sup} \var\{l_{\sigma_{\gamma}^2}(\bomega)-l_{\sigma_{\gamma}^2}(\dot\bomega)\}\le L,
		&&	\underset{\bomega\in \boldsymbol{N}}{\sup} \var\{l_{\xi_{4t}}(\bomega)-l_{\xi_{4t}}(\dot\bomega)\}\le L,
		&&	\underset{\bomega\in \boldsymbol{N}}{\sup} \var\{l_{\sigma_e^2}(\bomega)-l_{\sigma_e^2}(\dot\bomega)\}\le L,
	\end{align*}
	for $q=1,\ldots,p_a$, $r=1.\ldots,p_b$, $s=1,\ldots,p_{ab}$, $t=1,\ldots,p_w$.
\end{lma}
\begin{proof}
	
	For $g,h,m$ sufficiently large,  uniformly on $\bomega\in \boldsymbol{N}$, we have 
	\begin{align*}
		&mL_1 \le \lambda_1\le mU_1, &&mL_1\le \dlambda_1\le mU_1,\\
		&hmL_2 \le\lambda_2\le hmU_2, &&hm_L2 \le \dlambda_2\le hmU_2,\\
		&gmL_3 \le \lambda_3\le gmU_3, &&gmL_3 \le \dlambda_3\le gmU_3, \\
		&(g+h)mL_4 \le \lambda_4\le (g+h)mU_4, &&(g+h)mL_4 \le \dlambda_4\le (g+h)mU_4,
	\end{align*}
	where $U_i, L_i\ge 0$ are fixed constants for $i=1,2,3,4$. 
	In order to mitigate the proliferation of constant notation in the following calculation, we designate $M_0$ as a finite constant that is not necessarily the same at each appearance. Then we have 
	\begin{align*}
		&({1}/{\lambda_1}-{1}/{\dlambda_1})^2\le  (gh)^{-1}m^{-2}M_0,
		\qquad({1}/{\lambda_2}-{1}/{\dlambda_2})^2 \le g^{-1}h^{-2}m^{-2}M_0,\\
		&({1}/{\lambda_3}-{1}/{\dlambda_3})^2\le g^{-2}h^{-1}m^{-2}M_0,
		\qquad({1}/{\lambda_4}-{1}/{\dlambda_4})^2\le g^{-1}h^{-2}m^{-2}M_0 \text{ (or }\le g^{-2}h^{-1}m^{-2}M_0 ),\\
		&\var [(\alpha_i-\bar\alpha)+(\gammaid-\gammadd)+(\eidd-\eddd)]\le M_0,
		\qquad\var[(\beta_j-\bar\beta)+(\gammadj-\gammadd)+(\edjd-\eddd)] \le M_0,\\
		&\var[(\gamma_{ij}-\gammaid-\gammadj+\gammadd)+(\eijd-\eidd-\edjd+\eddd)]\le M_0,
		\qquad\var(\bar\alpha +\bar\beta+\bar\gamma+\bar{e})\le g^{-1}M_0 \text{ (or }\le h^{-1}M_0),\\
		&\var [\{(\alpha_i-\bar\alpha)+(\gammaid-\gammadd)+(\eidd-\eddd)\}^2]\le M_0,
		\qquad\var [\{(\beta_j-\bar\beta)+(\gammadj-\gammadd)+(\edjd-\eddd)\}^2]\le M_0,\\
		&\var [\{ (\gamma_{ij}-\gammaid-\gammadj+\gammadd)+(\eijd-\eidd-\edjd+\eddd)\}^2]\le M_0,
		\qquad\var[(e_{ijk}-\eijd)^2]\le M_0,\\
		&\var	[(\bar\alpha +\bar\beta+\bar\gamma+\bar{e})^2 ] \le g^{-2}M_0\text{ (or }\le h^{-2}M_0).
	\end{align*}
	For example, we have 
	\begin{align*}
		l_{\xi_0}(\bomega)-	l_{\xi_0}(\dot\bomega)&=\frac{n}{\lambda_4}[(\dot{\xi}_0-\xi_0)+\barbxat(\dot\bxi_1-\bxi_1)+\barbxbt(\dot\bxi_2-\bxi_2)+\barbxabt(\dot\bxi_3-\bxi_3)+\barbxwt(\dot\bxi_4-\bxi_4)]\\&
		+
		(\frac{n}{\lambda_4}-\frac{n}{\dot\lambda_4})(\bar\alpha +\bar\beta+\bar\gamma+\bar{e}),
	\end{align*}
	and so
	\begin{align*}
		\var[l_{\xi_0}(\bomega)-	l_{\xi_0}(\dot\bomega)]		=
		(\frac{n}{\lambda_4}-\frac{n}{\dot\lambda_4})^2\var(\bar\alpha +\bar\beta+\bar\gamma+\bar{e})\le M_0^2\le L.
	\end{align*}
The argument for each of the remaining term is similar and details can be found in Supplementary Section S.6.3.
\end{proof}
\subsection{Lemmas for the derivative function}\label{proof3}
To prove the main theorem, we use the mean value theorem to linearly approximate $\bpsi(\bomega)$. Each element of $\bpsi(\bomega)$ is evaluated using a different value of $\bomega$ in each row of the derivative matrix. We use a square matrix $\bOmega$ with dimensions $(p_a+p_b+p_{ab}+p_w+5)\times (p_a+p_b+p_{ab}+p_w+5)$ to evaluate the derivative matrix, denoted as $\triangledown\bpsi(\bOmega)$, where each row is evaluated at the corresponding row of $\bOmega$. We partition $\triangledown\bpsi(\bOmega)$ into sub-matrices corresponding to between-row, between-column, between-cell and within-cell parameters. These sub-matrices are $\bOmega^{(a)}$, $\bOmega^{(b)}$, $\bOmega^{(ab)}$, and $\bOmega^{(w)}$, respectively, with dimensions $(p_a+2)\times(p_a+p_b+p_{ab}+p_w+5)$, $(p_b+1)\times(p_a+p_b+p_{ab}+p_w+5)$, $(p_{ab}+1)\times(p_a+p_b+p_{ab}+p_w+5)$, and $(p_w+1)\times(p_a+p_b+p_{ab}+p_w+5)$. We write
\begin{displaymath}
	\triangledown\bpsi (\bOmega)=
	\begin{bmatrix}
		\triangledown\bpsi ^{(a),(a)}(\bOmega^{(a)})&\triangledown\bpsi ^{(a),(b)}(\bOmega^{(a)})&\triangledown\bpsi ^{(a),(ab)}(\bOmega^{(a)})&\triangledown\bpsi ^{(a),(w)}(\bOmega^{(a)})\\
		\triangledown\bpsi ^{(b),(a)}(\bOmega^{(b)})&\triangledown\bpsi ^{(b),(b)}(\bOmega^{(b)})&\triangledown\bpsi ^{(b),(ab)}(\bOmega^{(b)})&\triangledown\bpsi ^{(b),(w)}(\bOmega^{(b)})\\
		\triangledown\bpsi ^{(ab),(a)}(\bOmega^{(ab)})&\triangledown\bpsi ^{(ab),(b)}(\bOmega^{(ab)})&\triangledown\bpsi ^{(ab),(ab)}(\bOmega^{(ab)})&\triangledown\bpsi ^{(ab),(w)}(\bOmega^{(ab)})\\
		\triangledown\bpsi ^{(w),(a)}(\bOmega^{(w)})&\triangledown\bpsi ^{(w),(b)}(\bOmega^{(w)})&\triangledown\bpsi ^{(w),(ab)}(\bOmega^{(w)})&\triangledown\bpsi ^{(w),(w)}(\bOmega^{(w)})
	\end{bmatrix}.
\end{displaymath}
The off-diagonal submatrices, like $\triangledown\bpsi ^{(a),(b)}$ and $\triangledown\bpsi ^{(b),(a)}$, may have different arguments. However, they are transposes of each other when the arguments are the same. When all rows of $\bOmega$ are equal to $\bomega^T$, we simplify the notation by using $\bomega$ instead of $\bOmega$ and its submatrices, discarding the transpose, and using $\bomega$ as a generic symbol to represent any row of $\bOmega$ when the specific row is not important.

\begin{lma}\label{lma Bn}
	Suppose Condition A holds. Then as $g,h,m\to\infty$, $\|\bB-\bB_n\|=o(1)$, where   $\bB_n=-\bK^{-1/2}\E\triangledown\bpsi(\dot\bomega)\bK^{-1/2}$ and 
	\begin{align}\label{MatrixB}
\bB=\begin{bmatrix}
	\bB_{1(a),(a)}&\bB_{1(a),(b)}&\bzero_{[(p_{a}+2):(p_{ab}+1)]}&\bzero_{[(p_{a}+2):(p_{w}+1)]}\\
	\bB_{1(b),(a)}&\bB_{1(b),(b)}&\bzero_{[(p_{b}+1):(p_{ab}+1)]}&\bzero_{[(p_{b}+1):(p_{w}+1)]}\\
	\bzero_{[(p_{ab}+1):(p_{a}+2)]}&\bzero_{[(p_{ab}+1):(p_b+1)]}&\bB_2&\bzero_{[(p_{ab}+1):(p_{w}+1)]}\\
	\bzero_{[(p_{w}+1):(p_{a}+2)]}&\bzero_{[(p_{w}+1):(p_b+1)]}&\bzero_{[(p_{w}+1):(p_{ab}+1)]}&\bB_3
\end{bmatrix},
\end{align}
with 
\begin{align*}
&	\bB_{1(a),(a)}=
1/\dot{\tau}\begin{bmatrix}
	1&{\barbxat}&0 \\
	{\barbxa}&{\dot{\tau}\bD_1}/{\dot{\sigma}_{\alpha}^2}+{\barbxa\barbxat}&\bzero_{[p_a:1]}\\
	0&\bzero_{[1:p_a]}&{\dot{\tau}}/{(2\dot{\sigma}_\alpha^4)}
\end{bmatrix},
\hspace{1ex}	\bB_{1(a),(b)}=	\bB_{1(b),(a)}^T=1/\dot{\tau}
\begin{bmatrix}
	{\eta^{1/2}\barbxbt}&0\\
	{\eta^{1/2}\barbxa\barbxbt}&\bzero_{[p_a:1]}\\
	\bzero_{[1:p_b]}&0
\end{bmatrix},
\\
&	\bB_{1(b),(b)}=\begin{bmatrix}
	{\bD_2}/{\dot{\sigma}_{\beta}^2}+{\eta\barbxb\barbxbt}/{\dot{\tau}}&\bzero_{[p_b:1]}\\
	\bzero_{[1:p_b]}&{1}/{(2\dot{\sigma}_\beta^4)}
\end{bmatrix},
\hspace{1ex}	\bB_2=\begin{bmatrix}
	{\bD_3}/{\dot{\sigma}_{\gamma}^2}&\bzero_{[p_{ab}:1]}\\
	\bzero_{[1:p_{ab}]}&{1}/{(2\dot{\sigma}_\gamma^4)}
\end{bmatrix}
\hspace{1ex}\text{and}\hspace{1ex}
	\bB_3=\begin{bmatrix}
	{\bD_4}/{\dot{\sigma}_{e}^2}&\bzero_{[p_{w}:1]}\\
	\bzero_{[1:p_{w}]}&{1}{(2\dot{\sigma}_e^4)}
\end{bmatrix}.
\end{align*}
We have $\bA=\bB$ under normality, but not otherwise.  
\end{lma}
\begin{proof}
	The detailed proof is given in {Supplementary S.6.4}.
 \end{proof}
\begin{lma}\label{lma5}
	Suppose Condition A holds. Then, as $g,h,m\rightarrow \infty$,
	\begin{align*}
		\underset{\bomega\in \boldsymbol{N}}{\sup}|\bK^{-1/2} [\E\triangledown\bpsi(\bOmega)-\E\triangledown\bpsi(\dot\bomega)]\bK^{-1/2}|=o(1),
	\end{align*}
	where the rows of $\bOmega$ are possibly different but lie in between $\bomega$ and $\dot\bomega$.
\end{lma}
\begin{proof}
We write
\begin{equation}\label{eq:lema4}
	\begin{split}
			& \underset{\bomega\in \boldsymbol{N}}{\sup}|\bK^{-1/2} [\E\triangledown\bpsi(\bOmega)-\E\triangledown\bpsi(\dot\bomega)]\bK^{-1/2}|\\&
		\le
		\underset{\bomega\in \boldsymbol{N}}{\sup}\bigg(g^{-1}| \E\triangledown\bpsi^{(a),(a)}(\bOmega^{(a)})-\E\triangledown\bpsi^{(a),(a)}(\dot\bomega)|
		+  (gh)^{-1/2}| \E\triangledown\bpsi^{(a),(b)}(\bOmega^{(a)})-\E\triangledown\bpsi^{(a),(b)}(\dot\bomega)|\\&
		+g^{-1}h^{-1/2}| \E\triangledown\bpsi^{(a),(ab)}(\bOmega^{(a)})-\E\triangledown\bpsi^{(a),(ab)}(\dot\bomega)|
		+g^{-1/2}n^{-1/2}| \E\triangledown\bpsi^{(a),(w)}(\bOmega^{(a)})-\E\triangledown\bpsi^{(a),(w)}(\dot\bomega)|\\&
		+ (gh)^{-1/2}| \E\triangledown\bpsi^{(b),(a)}(\bOmega^{(b)})-\E\triangledown\bpsi^{(b),(a)}(\dot\bomega)|
		+h^{-1}| \E\triangledown\bpsi^{(b),(b)}(\bOmega^{(b)})-\E\triangledown\bpsi^{(b),(b)}(\dot\bomega)|\\&
		+g^{-1/2}h^{-1}| \E\triangledown\bpsi^{(b),(ab)}(\bOmega^{(b)})-\E\triangledown\bpsi^{(b),(ab)}(\dot\bomega)|
		+h^{-1/2}n^{-1/2}| \E\triangledown\bpsi^{(b),(w)}(\bOmega^{(b)})-\E\triangledown\bpsi^{(b),(w)}(\dot\bomega)|\\&
		+g^{-1}h^{-1/2}| \E\triangledown\bpsi^{(ab),(a)}(\bOmega^{(ab)})-\E\triangledown\bpsi^{(ab),(a)}(\dot\bomega)|
		+g^{-1/2}h^{-1}| \E\triangledown\bpsi^{(ab),(b)}(\bOmega^{(ab)})-\E\triangledown\bpsi^{(ab),(b)}(\dot\bomega)|\\&
		+ (gh)^{-1}| \E\triangledown\bpsi^{(ab),(ab)}(\bOmega^{(ab)})-\E\triangledown\bpsi^{(ab),(ab)}(\dot\bomega)|
		+ (gh)^{-1}m^{-1/2}|\E\triangledown\bpsi^{(ab),(w)}(\bOmega^{(ab)})-\E\triangledown\bpsi^{(ab),(w)}(\dot\bomega)|\\&
		+g^{-1/2}n^{-1/2}| \E\triangledown\bpsi^{(w),(a)}(\bOmega^{(w)})-\E\triangledown\bpsi^{(w),(a)}(\dot\bomega)|
		+h^{-1/2}n^{-1/2}| \E\triangledown\bpsi^{(w),(b)}(\bOmega^{(w)})-\E\triangledown\bpsi^{(w),(b)}(\dot\bomega)|\\&
		+ (gh)^{-1}m^{-1/2}| \E\triangledown\bpsi^{(w),(ab)}(\bOmega^{(w)})-\E\triangledown\bpsi^{(w),(ab)}(\dot\bomega)|
		+n^{-1}| \E\triangledown\bpsi^{(w),(w)}(\bOmega^{(w)})-\E\triangledown\bpsi^{(w),(w)}(\dot\bomega)|\bigg).
	\end{split}
\end{equation}
We then have, for example,
	\begin{align*}
		&\underset{\bomega\in \boldsymbol{N}}{\sup}g^{-1}| \E\triangledown\bpsi^{(a),(a)}(\bOmega^{(a)})-\E\triangledown\bpsi^{(a),(a)}(\dot\bomega)|\\&
		\le\underset{\bomega\in \boldsymbol{N}}{\sup}g^{-1}\sumig| \E\triangledown\bpsi_i^{(a),(a)}(\bOmega^{(a)})-\E\triangledown\bpsi_i^{(a),(a)}(\dot\bomega)|\\&
		\le \underset{\bomega\in \boldsymbol{N}}{\sup}g^{-1}\sumig 
		[|\E l_{i\xi_0\xi_0}-\E l_{i\xi_0\xi_0}(\dot\bomega )|+|\E l_{i\xi_0\bxi_1^T}-\E l_{i\xi_0\bxi_1^T}(\dot\bomega )|+|\E l_{i\xi_0\siga^2}-\E l_{i\xi_0\siga^2}(\dot\bomega )|\\&\quad
		+
		|\E l_{i\bxi_1\xi_0}-\E l_{i\bxi_1\xi_0}(\dot\bomega )|+|\E l_{i\bxi_1\bxi_1^T}-\E l_{i\bxi_1\bxi_1^T}(\dot\bomega )|+|\E l_{i\bxi_1\siga^2}-\E l_{i\bxi_1\siga^2}(\dot\bomega )|\\&\quad
		+
		|\E l_{i\siga^2\xi_0}-\E l_{i\siga^2\xi_0}(\dot\bomega )|+|\E l_{i\siga^2\bxi_1^T}-\E l_{i\siga^2\bxi_1^T}(\dot\bomega )|+|\E l_{i\siga^2\siga^2}-\E l_{i\siga^2\siga^2}(\dot\bomega )|]\\&
		=O(g^{-1/2}),
	\end{align*}
	because  $0\le \eta<\infty$,  $|hm/\dlambda_2-hm/\lambda_2|=O(g^{-1/2})$, $|hm/\dlambda_4-hm/\lambda_4|=O(g^{-1/2})$, $|\dot\xi_0-\xi_0|=O(g^{-1/2})$, $|\dot\bxi_1-\bxi_1|=O(g^{-1/2})$, and 
	$|\dot\bxi_2-\bxi_2|=O(h^{-1/2})$ uniformly on $\bomega\in \boldsymbol{N}$.
Similarly, the remaining terms inside the parentheses on the right side of equation (\ref{eq:lema4}) converge to zero at their corresponding rates. Details can be found in {Supplementary Section S.6.5}.
\end{proof}

\newcommand{\secderA}{\{(\alpha_i-\bar\alpha)+(\gammaid-\gammadd)+(\eidd-\eddd)\}}
\newcommand{\secderB}{\{(\beta_j-\bar\beta)+(\gammadj-\gammadd)+(\edjd-\eddd)\}}
\newcommand{\secderAB}{\{(\gamma_{ij}-\gammaid-\gammadj+\gammadd)+(\eijd-\eidd-\edjd+\eddd)\}}
\newcommand{\secderW}{(e_{ijk}-\eijd)}
\newcommand{\secderS}{(\bar\alpha+\bar\beta+\gammadd+\eddd)}
\newcommand{\secderXA}{\{  (\bxai-\barbxa)^T(\dot\bxi_1-\bxi_1)+(\barbxiab-\barbxab)^T(\dot\bxi_3-\bxi_3)+(\barbxiw-\barbxw)^T(\dot\bxi_4-\bxi_4)\}}
\newcommand{\secderXB}{\{  (\bxbj-\barbxb)^T(\dot\bxi_2-\bxi_2)+(\barbxjab-\barbxab)^T(\dot\bxi_3-\bxi_3)+(\barbxjw-\barbxw)^T(\dot\bxi_4-\bxi_4)\}}
\newcommand{\secderXAB}{\{(\bxabij-\barbxiab-\barbxjab+\barbxab)(\dot\bxi_3-\bxi_3)+(\barbxijw-\barbxiw-\barbxjw+\barbxw)(\dot\bxi_4-\bxi_4)\}}
\newcommand{\secderXW}{\{  (\bxwijk-\barbxijw)^T(\dot\bxi_4-\bxi_4)\}}
\newcommand{\secderXS}{\{  (\dot\xi_0-\xi_0)+\barbxat(\dot\bxi_1-\bxi_1)+\barbxbt(\dot\bxi_2-\bxi_2)+(\barbxiab-\barbxab)^T(\dot\bxi_3-\bxi_3)+(\barbxiw-\barbxw)^T(\dot\bxi_4-\bxi_4)\}}

\begin{lma}\label{lma6}
	Suppose Condition A holds, As, $g,h,m\to\infty$,
	\begin{align*}
		\underset{\bomega\in \boldsymbol{N}}{\sup} g^{-1/4}\|\bK^{-1/2}\{\triangledown\bpsi(\bOmega)-\E \triangledown\bpsi(\bOmega)\}\bK^{-1/2}\|=o_p(1).
	\end{align*}
\end{lma}
\begin{proof}
	The non-zero elements in the $(p_a+2)\times(p_a+2)$ matrix $\triangledown\bpsi^{(a),(a)}(\bomega)-\E \triangledown\bpsi^{(a),(a)}(\bomega)$ are
	\begin{align*}
		&|l_{\xi_0\siga^2}(\bomega)-\E l_{\xi_0\siga^2}(\bomega)|=O_p(g),
		\hspace{1ex} |l_{\siga^2\siga^2}(\bomega)-\E l_{\siga^2\siga^2}(\bomega)|
		=O_p(g)\\&
  \text{and}
 \hspace{1ex} |l_{\xi_{1q}\siga^2}(\bomega)-\E l_{\xi_{1q}\siga^2}(\bomega)|=O_p(g),
 \hspace{1ex} \text{for} \hspace{1ex} q=1,\ldots,p_a.
	\end{align*}
	Accordingly, it is enough to show the uniform convergence to zero of the elements of 
	$g^{-5/4}\{\triangledown\bpsi^{(a),(a)}(\bOmega^{(a)})-\E \triangledown\bpsi^{(a),(a)}(\bOmega^{(a)})\}$.
	
	Similarly, we can also obtain the uniform convergence to zero of the remaining elements of  
	$g^{-3/4}h^{-1/2}\{\triangledown\bpsi^{(a),(b)}(\bOmega^{(a)})-\E \triangledown\bpsi^{(a),(b)}(\bOmega^{(a)})\}$,
	$g^{-5/4}h^{-1/2}\{\triangledown\bpsi^{(a),(ab)}(\bOmega^{(a)})-\E \triangledown\bpsi^{(a),(ab)}(\bOmega^{(a)})\}$,
	$g^{-3/4}n^{-1/2}\{\triangledown\bpsi^{(a),(w)}(\bOmega^{(a)})-\E \triangledown\bpsi^{(a),(w)}(\bOmega^{(a)})\}$,
	$g^{-1/4}h^{-1}\{\triangledown\bpsi^{(b),(b)}(\bOmega^{(b)})-\E \triangledown\bpsi^{(b),(b)}(\bOmega^{(b)})\}$,
	$g^{-3/4}h^{-1}\{\triangledown\bpsi^{(b),(ab)}(\bOmega^{(b)})-\E \triangledown\bpsi^{(b),(ab)}(\bOmega^{(b)})\}$,
	$g^{-3/4}h^{-1}m^{-1/2}\{\triangledown\bpsi^{(b),(w)}(\bOmega^{(b)})-\E \triangledown\bpsi^{(b),(w)}(\bOmega^{(b)})\}$,
	$g^{-5/4}h^{-1}\{\triangledown\bpsi^{(ab),(ab)}(\bOmega^{(ab)})-\E \triangledown\bpsi^{(ab),(ab)}(\bOmega^{(ab)})\}$,
	$g^{-5/4}h^{-1}m^{-1/2}\{\triangledown\bpsi^{(ab),(w)}(\bOmega^{(ab)})-\E \triangledown\bpsi^{(ab),(w)}(\bOmega^{(ab)})\}$
	and
	$g^{-1/4}n^{-1}\{\triangledown\bpsi^{(w),(w)}(\bOmega^{(w)})-\E \triangledown\bpsi^{(w),(w)}(\bOmega^{(w)})\}$. Details can be found in {Supplementary Section S.6.6}.
\end{proof}


\section{Discussion}\label{sec:7}
We have established asymptotic theory for the two-way crossed random effect with interaction model. We imposed mild conditions to derive the asymptotic properties of its estimators. Importantly, we did not constrain the rates for \( g \), \( h \), and \( m \) to diverge to infinity, and we assumed moment conditions rather than normality. 
The main challenges in this research were to calculate the inverse of \( \mathbf{V} \), to obtain tractable expressions for the likelihood estimating equations, and to group the parameters into four clear groups (row, column, interaction, and within-cell parameters) corresponding to different levels of information in the data. This approach eliminates the need for rate restrictions on \( g \), \( h \), and \( m \), and also produces an elegantly structured asymptotic covariance matrix.

The scope of this paper is confined to the two-way crossed random effect with interaction model, which serves as a canonical example of crossed random effect models, featuring both the crossed structure and interaction terms. Our analysis of this model's asymptotic properties sets the stage for broadening these approaches to a wider variety of crossed effect models.
We addressed the balanced case to use the simplifications afforded by the Kronecker product notation in computation and derivation. Despite this simplification, the calculations remain challenging, with all details relegated to the supplementary material. 
Expanding our approach to the unbalanced case follows the same principles, though the resulting expressions and derivations become even more complicated. Our theoretical and simulation results suggest that with \( g \) and \( h \) increasing indefinitely, we can achieve asymptotic normality for all estimators. Importantly, our framework accommodates an infinite increase in \( m \), a necessary condition for the consistent prediction of random effects as outlined by \cite{jiang1998asymptotic}.
For instance, \cite{lyu2023small,lyu2019eblup} have shown that \( m \) tending to infinity is required to establish asymptotic normality for the best linear unbiased predictor (EBLUP) and for model-based small area estimators. Predictions of random effects for the two-way crossed random effect with interaction model will be pursued in future work. 

%
\begin{appendix}

\section{Notation }\label{appA}
We define 4 means for each variable with 3 subscripts,
\begin{align*}
&\yddd=n^{-1}\sumig\sumjh\sumkm y_{ijk},	&&\barbxw=n^{-1}\sumig\sumjh\sumkm\bxwijk, &&\eddd=n^{-1}\sumig\sumjh\sumkm e_{ijk},
	\\ 
	&\yijd=m^{-1}\sumkm y_{ijk}, &&\yidd=(hm)^{-1}\sumjh\sumkm  y_{ijk},&&\ydjd=(gm)^{-1}\sumig \sumkm  y_{ijk}, 
	\\	
	&\barbxijw=m^{-1}\sumkm\bxwijk,
	&&\barbxiw=(hm)^{-1}\sumjh\sumkm\bxwijk, &&\barbxjw=(gm)^{-1}\sumig\sumkm\bxwijk, 
	\\
	&\eijd=m^{-1}\sumkm e_{ijk}, 	
	&&\eidd=(hm)^{-1}\sumjh\sumkm e_{ijk},&&\edjd=(gm)^{-1}\sumig\sumkm e_{ijk},
\end{align*}
and then 3 means for each variable with 2 subscripts,
\begin{align*}
&\barbxiab=h^{-1}\sumjh\bxabij,
	&&\barbxjab=g^{-1}\sumig\bxabij, 
	&&\barbxab=(gh)^{-1}\sumig\sumjh\bxabij,
 \\
	 	& \gammaid=h^{-1}\sumjh\gamma_{ij},
	&&\gammadj=g^{-1}\sumig\gamma_{ij},
	&&\gammadd=(gh)^{-1}\sumig\sumjh\gamma_{ij}.
\end{align*}
Finally for the variables with a single subscript, we define
\begin{align*}
&\barbxa=g^{-1}\sumig\bxai, &&\barbxb=h^{-1}\sumjh\bxbj, 
	 	&&\bar\alpha=g^{-1}\sumig\alpha_i, &&\bar\beta=h^{-1}\sumjh \beta_j.
\end{align*}
We denote 
\begin{align*}
& \cyijkw=y_{ijk}-\yijd, &&\cyijab=\yijd-\yidd-\ydjd+\yddd, &&\cyia=\yidd-\yddd, 	\\
&\cxijkw=\bxwijk-\barbxijw, &&\cxijw=\barbxijw-\barbxiw-\barbxjw+\barbxw, &&\cxiw=\barbxiw-\barbxw,\\
& &&\cxijab=\bxabij-\barbxiab-\barbxjab+\barbxab,
&&\cxiab=\barbxiab-\barbxab,
\\
& && &&\cxia=\bxai-\barbxa,\\
\end{align*}
with the single subscript $j$ terms defined analogously to the subscript $i$ terms as
\begin{align*}
&\cyjb=\ydjd-\yddd, &&\cxjw=\barbxjw-\barbxw, &&\cxjab=\barbxjab-\barbxab, &&\cxjb=\bxbj-\barbxb.
\end{align*}

We also define some notation for sum squares:
\begin{enumerate}
	\item Sums of squares for factor A:
	$SA_{(a)}^x=\sumig\cxia\cxiat$,
	$SA_{(a)}^{xy}=\sumig\cxia\cyia$,
	$SA_{(a),(ab)}^x=\sumig\cxia\cxiabt$,
	$SA_{(a),(w)}^x=\sumig\cxia\cxiwt$, 
	$SA_{(ab)}^x=\sumig\cxiab\cxiabt$,
	$SA_{(ab)}^{xy}=\sumig\cxiab\cyia$,
	$SA_{(ab),(w)}^x=\sumig\cxiab\cxiwt$,
	$SA_{(w)}^x=\sumig\cxiw\cxiwt$,
	$SA_{(w)}^{xy}=\sumig\cxiw\cyia$.
	\item Sums of squares for factor B:
	$SB_{(b)}^x=\sumjh\cxjb\cxjbt$,
	$SB_{(b)}^{xy}=\sumjh\cxjb\cyjb$,
	$SB_{(b),(ab)}^x=\sumjh\cxjb\cxjabt$,
	$SB_{(b),(w)}^x=\sumjh\cxjb\cxjwt$, 
	$SB_{(ab)}^x=\sumjh\cxjab\cxjabt$,
	$SB_{(ab)}^{xy}=\sumjh\cxjab\cyjb$,
	$SB_{(ab),(w)}^x=\sumjh\cxjab\cxjwt$,
	$SB_{(w)}^x=\sumjh\cxjw\cxjwt$,
	$SB_{(w)}^{xy}=\sumjh\cxjw\cyjb$.
	\item Sums of squares for interactions:
	$SAB_{(ab)}^x=\sumig\sumjh\cxijab\cxijabt$,\\
	$SAB_{(ab)}^{xy}=\sumig\sumjh\cxijab\cyijab$,
	$SAB_{(ab),(w)}^{x}=\sumig\sumjh\cxijab\cxijwt$,\\
	$SAB_{(w)}^{x}=\sumig\sumjh\cxijw\cxijwt$,
	$SAB_{(w)}^{xy}=\sumig\sumjh\cxijw\cyijab$.
	\item Sums squares within cells:
	$SW_{(w)}^x=\sumig\sumjh\sumkm\cxijkw\cxijkwt$ and\\
	$SW_{(w)}^{xy}=\sumig\sumjh\sumkm\cxijkw\cyijkw$.
\end{enumerate}

\section{Inverse of $\bV$}\label{appendix-1}\label{AppendixB}
Following the algorithm of \citep{searle1979dispersion}, we define
\begin{align*}
	\bT_3=\left[ \begin{matrix}
		1&0\\1&g
	\end{matrix}\right] \otimes\left[ \begin{matrix}
		1&0\\1&h
	\end{matrix}\right]\otimes \left[ \begin{matrix}
		1&0\\1&m
	\end{matrix}\right].
\end{align*}
Then we let $\blambda=\bT_3\btheta$, so that
\begin{align*}
	\blambda=
	\left[\begin{matrix}
		1&0&0&0&0&0&0&0\\
		1&m&0&0&0&0&0&0\\
		1&0&h&0&0&0&0&0\\
		1&m&h&hm&0&0&0&0\\
		1&0&0&0&g&0&0&0\\
		1&m&0&0&g&gm&0&0\\
		1&0&h&0&g&0&gh&0\\
		1&m&h&hm&g&gm&gh&ghm
	\end{matrix}	
	\right] 
	\left[ \begin{matrix}
		\sigma_e^2\\\sigma_\gamma^2\\0\\\sigma_\alpha^2\\0\\\sigma_\beta^2\\0\\0		
	\end{matrix}
	\right]=
	\left[\begin{matrix}
		\sigma_e^2\\
		\sigma_e^2+m\sigma_\gamma^2\\
		\sigma_e^2\\
		\sigma_e^2+m\sigma_\gamma^2+hm\sigma_\alpha^2\\
		\sigma_e^2\\
		\sigma_e^2+m\sigma_\gamma^2+gm\sigma_\beta^2\\
		\sigma_e^2\\
		\sigma_e^2+m\sigma_\gamma^2+hm\sigma_\alpha^2+gm\sigma_\beta^2
	\end{matrix}
	\right] =\left[ \begin{matrix}
		\lambda_0\\\lambda_1\\\lambda_0\\\lambda_2\\\lambda_0\\\lambda_3\\\lambda_0\\\lambda_4
	\end{matrix}\right],
\end{align*}
where 
\begin{align*}
	\lambda_0&=\sigma_e^2,&\lambda_1&=\sigma_e^2+m\sigma_\gamma^2,
	&\lambda_2&=\sigma_e^2+m\sigma_\gamma^2+hm\sigma_\alpha^2,\\
	\lambda_3&=\sigma_e^2+m\sigma_\gamma^2+gm\sigma_\beta^2
	\quad\text{and}&
	\lambda_4&=\sigma_e^2+m\sigma_\gamma^2+hm\sigma_\alpha^2+gm\sigma_\beta^2.
\end{align*}

With $\bnu$ having elements that are reciprocals of the elements of $\blambda$, the equation $\boldsymbol{{\mathsf{t}}}=\bT^{-1}\bnu$ gives
\begin{align*}
	\boldsymbol{{\mathsf{t}}}&=\frac{1}{ghm}
	\left[\begin{matrix}
		g&0\\-1&1
	\end{matrix} \right] 
	\otimes
	\left[\begin{matrix}
		h&0\\-1&1
	\end{matrix} \right] \otimes
	\left[\begin{matrix}
		m&0\\-1&1
	\end{matrix} \right] \bnu\\
	&=\frac{1}{ghm}
	\left[ \begin{matrix}
		ghm&0&0&0&0&0&0&\\
		-gh&gh&0&0&0&0&0&0\\
		-gm&0&gm&0&0&0&0&0\\
		g&-g&-g&g&0&0&0&0\\
		-hm&0&0&0&hm&0&0&0\\
		h&-h&0&0&-h&h&0&0\\
		m&0&-m&0&-m&0&m&0\\
		-1&1&1&-1&1&-1&-1&1
	\end{matrix}\right] 
	\left[ \begin{matrix}
		\frac{1}{\lambda_0}\\\frac{1}{\lambda_1}\\\frac{1}{\lambda_0}\\\frac{1}{\lambda_2}\\\frac{1}{\lambda_0}\\\frac{1}{\lambda_3}\\\frac{1}{\lambda_0}\\\frac{1}{\lambda_4}
	\end{matrix}\right] =
	\left[ \begin{matrix}
		\mathsf{t}_{000}\\\mathsf{t}_{001}\\0\\\mathsf{t}_{011}\\0\\\mathsf{t}_{101}\\0\\\mathsf{t}_{111}\\
	\end{matrix}\right],
\end{align*}
where
\begin{flalign*}
	&\mathsf{t}_{000}=\frac{1}{\lambda_0},&&\mathsf{t}_{001}=\frac{1}{m}(\frac{1}{\lambda_1}-\frac{1}{\lambda_0}),
	&&\mathsf{t}_{011}=\frac{1}{hm}(\frac{1}{\lambda_2}-\frac{1}{\lambda_1}),
	\\&\mathsf{t}_{101}=\frac{1}{gm}(\frac{1}{\lambda_3}-\frac{1}{\lambda_1}),
	&&\mathsf{t}_{111}=\frac{1}{ghm}(\frac{1}{\lambda_1}-\frac{1}{\lambda_2}-\frac{1}{\lambda_3}+\frac{1}{\lambda_4}).
\end{flalign*}
Then the inverse of  $\bV$ is given by
\begin{align*}
	\bV^{-1}=&\mathsf{t}_{000}(\bI_g\otimes\bI_h\otimes\bI_m)+\mathsf{t}_{001}(\bI_g\otimes\bI_h\otimes\bJ_m)+
	\mathsf{t}_{011}(\bI_g\otimes\bJ_h\otimes\bJ_m)\\&
 +\mathsf{t}_{101}(\bJ_g\otimes\bI_h\otimes\bJ_m)+\mathsf{t}_{111}(\bJ_g\otimes\bJ_h\otimes\bJ_m).
\end{align*}
Substituting  the $\mathsf{t}$s with $\lambda$s,  we rewrite the inverse of $\bV$ as 
\begin{align*}
	\bV^{-1}&=\mathsf{t}_{000}(\bI_g\otimes\bI_h\otimes\bI_m)+\mathsf{t}_{001}(\bI_g\otimes\bI_h\otimes\bJ_m)+
	\mathsf{t}_{011}(\bI_g\otimes\bJ_h\otimes\bJ_m)\\&\quad
	+\mathsf{t}_{101}(\bJ_g\otimes\bI_h\otimes\bJ_m)+\mathsf{t}_{111}(\bJ_g\otimes\bJ_h\otimes\bJ_m)\\&
	=\frac{1}{\lambda_0}(\bI_g\otimes\bI_h\otimes\bI_m)+\frac{1}{m}(\frac{1}{\lambda_1}-\frac{1}{\lambda_0})(\bI_g\otimes\bI_h\otimes\bJ_m)
	+\frac{1}{hm}(\frac{1}{\lambda_2}-\frac{1}{\lambda_1})(\bI_g\otimes\bJ_h\otimes\bJ_m)\\&\qquad
	+\frac{1}{gm}(\frac{1}{\lambda_3}-\frac{1}{\lambda_1})(\bJ_g\otimes\bI_h\otimes\bJ_m)
	+\frac{1}{ghm}(\frac{1}{\lambda_1}-\frac{1}{\lambda_2}-\frac{1}{\lambda_3}+\frac{1}{\lambda_4})(\bJ_g\otimes\bJ_h\otimes\bJ_m)\\&
	=\frac{1}{\lambda_0}(\bI_g\otimes\bI_h\otimes\bC_m)+\frac{1}{\lambda_1}(\bC_g\otimes\bC_h\otimes\bar{\bJ}_m)+\frac{1}{\lambda_2}(\bC_g\otimes\bar{\bJ}_h\otimes\bar{\bJ}_m)+\frac{1}{\lambda_3}(\bar{\bJ}_g\otimes\bC_h\otimes\bar{\bJ}_m)\\&\qquad
	+\frac{1}{\lambda_4}(\bar{\bJ}_g\otimes\bar{\bJ}_h\otimes\bar{\bJ}_m),
\end{align*}
where $\bar{\bJ}_g=\frac{1}{g}{\bJ}_g$ and $\bC_g=\bI_g-\bar\bJ_g$.  We have the identities 
\begin{align*}\label{iden}
	\tr(\bar{\bJ}_g)=1,\qquad\tr(\bC_g)=g-1,\qquad\bC_g\bJ_g=\bzero , \qquad \bar{\bJ}_g\bJ_g=\bJ_g,\qquad\bone_g^T\bC_g=\bzero,\qquad\bar\bJ_g\bone_g=\bone_g.
\end{align*}
 
\end{appendix}


\begin{funding}
SAS and AHW are respectively supported by the Australian Research Council Discovery Projects DP220103269 and DP230101908.
\end{funding}

\begin{supplement}
\stitle{Additional Information: Calculations and Proofs}
\sdescription{Comprehensive calculations and detailed proofs support the findings presented in the main manuscript.}
\end{supplement}
\begin{supplement}
\stitle{Simulation Resources: Code and Data Tables}
\sdescription{This part includes the R code used for simulations.}
\end{supplement}
 


	\bibliographystyle{imsart-number}
\bibliography{mythesisbib}
\end{document}

%% file: definition.tex
\newtheorem{theoremm}{Theorem} 

\newtheorem{corollary}[theoremm]{Corollary}
\newtheorem{lma}{Lemma}
\newcommand{\E}{\operatorname{E}}
\newcommand{\pr}{\operatorname{Pr}}

\newcommand{\var}{\operatorname{Var}}

\newcommand{\tr}{\operatorname{tr}}
\newcommand{\diag}{\operatorname{diag}}
\newcommand{\bA}{\mathbf{A}}
\newcommand{\bB}{\mathbf{B}}
\newcommand{\bC}{\mathbf{C}}
\newcommand{\bD}{\mathbf{D}}

\newcommand{\bF}{\mathbf{F}}

\newcommand{\bI}{\mathbf{I}}
\newcommand{\bJ}{\mathbf{J}}
\newcommand{\bK}{\mathbf{K}}

\newcommand{\bT}{\mathbf{T}}

\newcommand{\bV}{\mathbf{V}}

\newcommand{\bX}{\mathbf{X}}

\newcommand{\bZ}{\mathbf{Z}}
\newcommand{\ba}{\mathbf{a}}

\newcommand{\be}{\mathbf{e}}
\newcommand{\bbf}{\mathbf{f}}

\newcommand{\bo}{\mathbf{o}}

\newcommand{\bx}{\mathbf{x}}
\newcommand{\by}{\mathbf{y}} 

\newcommand{\dlambda}{\dot{\lambda}}

\newcommand{\balpha}{\boldsymbol{\alpha}}
\newcommand{\btheta}{\boldsymbol{\theta}}
\newcommand{\bomega}{\boldsymbol{\omega}}
\newcommand{\blambda}{\boldsymbol{\lambda}}
\newcommand{\bgamma}{\boldsymbol{\gamma}}

\newcommand{\bpsi}{\boldsymbol{\psi}}
\newcommand{\bphi}{\boldsymbol{\phi}}
\newcommand{\bbeta}{\boldsymbol{\beta}}

\newcommand{\bnu}{\boldsymbol{\nu}}
\newcommand{\bzero}{\boldsymbol{0}}
\newcommand{\bone}{\boldsymbol{1}}
\newcommand{\bxi}{\boldsymbol{\xi}}

\newcommand{\hbtheta}{\hat{\btheta}}
\newcommand{\hbomega}{\hat{\bomega}}

\newcommand{\bOmega}{\boldsymbol{\Omega}}

\newcommand{\siga}{{\sigma}_{\alpha}}  
\newcommand{\sigb}{{\sigma}_{\beta}}  
\newcommand{\siggama}{{\sigma}_{\gamma}}  
\newcommand{\sige}{{\sigma}_{e}}  

\newcommand{\dsiga}{\dot{\sigma}_{\alpha}}  
\newcommand{\dsigb}{\dot{\sigma}_{\beta}}  
\newcommand{\dsiggama}{\dot{\sigma}_{\gamma}}  
\newcommand{\dsige}{\dot{\sigma}_{e}}